\pdfoutput=1
 
\documentclass[twoside,11pt,a4paper,fleqn]{scrartcl}

\usepackage[ngerman, english]{babel} 
\usepackage[latin1]{inputenc}
\usepackage{csquotes} 
\usepackage[T1]{fontenc} 
\usepackage{amsmath} 
\usepackage{amssymb} 
\usepackage[automark,
headsepline=0.1pt
]{scrlayer-scrpage} 
\ihead{}
\cehead{{\footnotesize \uppercase{DMPC for continuous-time nonlinear systems based on suboptimal ADMM}}}
\cohead{{\footnotesize \uppercase{Anja Bestler and Knut Graichen}}}
\ohead{\pagemark}
\ofoot[]{} 
\ifoot[]{}
\usepackage[url=false,doi=false,eprint=true,isbn=false,backend=bibtex,style=numeric]{biblatex}
\usepackage{graphicx} 
\usepackage[bf]{caption}	
\usepackage{xcolor}
\usepackage{pgfplots}				
\pgfplotsset{compat=newest}
\usetikzlibrary{plotmarks}
\usetikzlibrary{arrows.meta}
\usepgfplotslibrary{patchplots}
\usepackage{import}

\usepackage{enumitem} 
\usepackage{amsthm} 
\newtheoremstyle{mytheoremstyle} 
{} 
{} 
{\itshape} 
{} 
{\bfseries} 
{.} 
{\labelsep} 
{} 
\theoremstyle{mytheoremstyle}	
\makeatletter
\renewenvironment{proof}[1][\proofname]{\par
	\pushQED{\qed}%
	\normalfont \topsep6\p@\@plus6\p@\relax
	\trivlist
	\item\relax
	{\bfseries
		#1\@addpunct{.}}\hspace\labelsep\ignorespaces
}{%
	\popQED\endtrivlist\@endpefalse
}
\makeatother

\usepackage[
left=20mm,
right=20mm,
bottom=30mm,
top=30mm]
{geometry} 

\newtheorem{assumption}{Assumption}
\newtheorem{corollary}{Corollary}
\newtheorem{lemma}{Lemma}
\newtheorem{theorem}{Theorem}

\usepackage{ushort}
\newcommand{\ubar}{\ushort}
\newcommand{\xsys}{\ubar x}
\newcommand{\usys}{\ubar u}
\newcommand{\xmpc}{x}
\newcommand{\umpc}{u}
\newcommand{\dt}{\Delta t}
\newcommand{\Nbb}{\mathbb{N}}
\newcommand{\Rbb}{\mathbb{R}}

\newcommand{\Mcal}{\mathcal{M}}
\newcommand{\Ncal}{\mathcal{N}}

\newcommand{\Vcal}{\mathcal{V}}
\newcommand{\Wcal}{\mathcal{W}}
\newcommand{\Zcal}{\mathcal{Z}}
\newcommand{\TT}			{\textsf{T}}
\newcommand{\dd}			{\;\mathrm{d}}

\newcommand{\norm}			[1]{  \left\| #1 \right\| }
\newcommand{\supnorm}		[1]{  \norm{#1}_{L^{\infty}} }

\newcommand{\Nto}[1]{\Ncal_{#1\to}}
\newcommand{\Ngets}[1]{\Ncal_{#1\gets}}

\DeclareMathOperator{\subjto}{s.t.}
\DeclareMathOperator{\diag}{diag}
\DeclareMathOperator{\mysup}{sup}
\DeclareMathOperator{\mylim}{lim}
\DeclareMathOperator{\mymin}{min}
\DeclareMathOperator{\mymax}{max}
\DeclareMathOperator{\mylog}{log}
\DeclareMathOperator{\myarg}{arg}
\DeclareMathOperator{\avg}{avg}

\newlength\fheight
\newlength\fwidth
\begin{document}
\thispagestyle{scrplain}

\begin{center}
{\Large\bfseries
Distributed model predictive control for continuous-time
nonlinear systems based on suboptimal ADMM\,\footnotemark}
\footnotetext{
	This work was funded by the Deutsche Forschungsgemeinschaft
	(DFG, German Research Foundation) under project no. GR 3870/4-1.
}
\\[2ex]
{\large Anja Bestler\footnotemark and Knut Graichen}
\\[2ex]
Institute for Measurement, Control and Microtechnology,
Ulm University,\\
Albert-Einstein-Allee 41, 89081 Ulm, Germany
\end{center}

\footnotetext{
	Correspondence to: Anja Bestler, Institute of Measurement, Control, and Microtechnology, Ulm University, Albert-Einstein-Allee 41, 89081 Ulm, Germany, Email: anja.bestler@uni-ulm.de
}

\textbf{Abstract:} The paper presents a distributed model predictive
control (DMPC) scheme for continuous-time nonlinear systems based on the
alternating direction method of multipliers (ADMM). A stopping
criterion in the ADMM algorithm limits the iterations and therefore
the required communication effort during the distributed MPC
solution at the expense of a suboptimal solution. Stability results
are presented for the suboptimal DMPC scheme under two different ADMM
convergence assumptions. In particular, it is shown that the required iterations
in each ADMM step are bounded, which is also confirmed in simulation studies. 


\section{Introduction}

Model predictive control (MPC) has become a popular control method due
to its applicability for linear and nonlinear systems,
the ability to handle constraints, and its optimal performance in terms
of the cost function to be
minimized~\cite{Mayne:AUT:2000,Camacho:book:2003,Mattingley:CS:2011}. 
Recently, the research on MPC for large-scale systems has led to the
concept of distributed model predictive control (DMPC), see e.g.\
\cite{Camponogara:CST:2002,Scattolini:JPC:2009,Negenborn:book:2014}, where
subsystems of the global systems are controlled by local MPCs.
The union of subsystem and controller is commonly called agent. If the
MPC agents are allowed to exchange information among each other and apply
a suitable coordination scheme, they can cooperatively solve the
optimal control problem (OCP) underlying the MPC formulation
for the complete system. The motivation for this multi-agent setup of
the controller lies within its flexibility, since adding or removing
agents does not require a complete redesign of a central MPC.

Dynamically coupled systems have been in the focus of
DMPC research over the last years. Typically, the coupling between
these systems is given by physical interconnection and cover a wide
range of applications, for instance, cooperative payload transport in
robotic applications \cite{Fink:IJRR:2011} or transportation systems
such as smart grids~\cite{Venkat:CST:2008,Stadler:SYSCON:2016} 
and water distribution
networks~\cite{Leirens:ACC:2010,Ocampo-Martinez:IET:2012,Hentzelt:MSC:2014}. 
There exists a variety of DMPC schemes in the literature
that can be structured in terms of the considered system
class and coupling, stability proporties and algorithmic considerations.

The majority of the literature on DMPC for dynamically coupled systems
considers linear, discrete-time systems, see e.g.\
\cite{Giselsson:Springer:2014,Zeilinger:CDC:2013,Mueller:NMPC:2012,Koegel:DYCOPS:2013,Farina:CST:2014},
though continuous-time as well as nonlinear systems have been 
in the focus of the recent past as well
\cite{Farokhi:Springer:2014,Mueller:IJRNC:2012,Kozma:Springer:2014,Farina:SCL:2014}.  
In most cases, the subsystems are assumed to be linearly coupled
to the state and/or control variables of their neighbors
\cite{Giselsson:Springer:2014,Zeilinger:CDC:2013,Koegel:DYCOPS:2013,Farina:SCL:2014}.
Alternative formulations
concern linear output variables or functions 
\cite{Kozma:Springer:2014,Farina:CST:2014} or coupling
constraints~\cite{Mueller:NMPC:2012,Mueller:IJRNC:2012}.

A common approach to ensure stability of a DMPC scheme for dynamically
coupled systems is to use local terminal costs and terminal set
constraints, which in connection with local controllers are rendered
invariant and lead to a reduction of the terminal costs in the sense
of control Lyapunov functions (CLF)
\cite{Zeilinger:CDC:2013,Mueller:NMPC:2012,Mueller:IJRNC:2012,Koegel:DYCOPS:2013}. 
A further approach to guarantee stability is relaxed dynamic
programming (RDP), where the existence of a control law in connection
with a suitable descent condition is assumed \cite{Gruene:TAC:2008},
that can be extended to DMPC as
well~\cite{Farokhi:Springer:2014,Giselsson:Springer:2014}. 
From a practical point of view, an important issue directly
linked to stability is the tradeoff to be made between communication
effort and control performance.
For instance, a limited number of communication steps within a DMPC
scheme will lead to a suboptimal solution, for which stability still
has to be ensured.
Moreover, terminal set constraints that are often assumed
for stability are unfavorable from a numerical viewpoint, leading 
to an increased computational load for the single
MPC agents compared to an MPC formulation without terminal
constraints~\cite{Limon:TAC:2006,Graichen:TAC:2010}.

The numerical realization of DMPC requires a suitable 
decomposition scheme to allow for a parallel solution on the single
MPC agents. One approach is to discretize the high-dimensional
MPC optimization problem and use tailored parallelization schemes
\cite{Talukdar:CDC:2005,Doan:JPC:2011,Necoara:CDC:2008,Necoara:CDC:2009}.
These contributions mainly consider linear, discrete-time dynamics
due to their favorable decomposability over nonlinear
systems, which so far have been considered only sporadically
\cite{Necoara:CDC:2009,Kozma:Springer:2014}. 
Another numerical realization of DMPC for coupled dynamical systems
concerns primal decomposition, where a distributed solution is
realized by local sensitivity-based approximations of the objective
function of the neighboring
subsystems~\cite{Venkat:CDC:2005,Venkat:CST:2008,Scheu:JPC:2011}. 
The primal decomposition approach takes advantage of the superposition
principle of linear systems, though certain extensions to
nonlinear systems exist~\cite{Scheu:ACC:2010,Stewart:JPC:2011}.
A drawback of primal decomposition is that all local agents require
knowledge of the overall system dynamics (an exception
is~\cite{Scheu:JPC:2011}), which restricts the 
flexibility and scalability of the DMPC scheme.

Another class of decomposition schemes concerns dual decomposition,
where coupling constraints are taken into account by multipliers. The
resulting formulation consists of a sum of cost functions that are
coupled via equality constraints~\cite{Rantzer:ACC:2009} and enable a
parallel solution. The single subproblems are coordinated by
the update of the multipliers serving as dual variables in the
optimization. A variety of DMPC schemes relies on dual decomposition,
see e.g.\
\cite{Cheng:JPC:2005,Negenborn:INCOM:2006,Necoara:CDC:2009,Giselsson:CDC:2010,Doan:JPC:2011,McNamara:IFACWC:2011}. 
Recent contributions in this field also consider extensions
of dual decomposition to improve convergence, such as accelerated gradient
methods~\cite{Giselsson:Springer:2014} and in particular
ADMM (Alternating Direction Method of Multipliers)
\cite{Koegel:ACCP:2012,Farokhi:Springer:2014}.

This contribution presents a DMPC scheme for nonlinear,
continuous-time systems, where the coupling of the subsystems is given
by the state variables of its neighbors. The centralized MPC problem
is formulated without terminal set constraints as the basis for an
efficient numerical solution. The DMPC scheme is realized by dual
decomposition and ADMM. A suitable stopping criterion is introduced to
limit the ADMM iterations and therefore the 
communication effort between the local MPC agents. Besides the
neighbor-to-neighbor communication between the ADMM iterates, the
stopping criterion requires one global communication step or
alternatively the definition of a master agent. 
The premature stopping of the ADMM algorithm leads to a suboptimal
solution, where the remaining residual of the consistency conditions
can be interpreted as the optimization error in each DMPC step.
Assuming a linear convergence property for the ADMM algorithm itself,
asymptotic stability as well as exponential decay of the optimization
error is shown. 

Though linear convergence of ADMM has recently been proved for
different ADMM settings in finite and infinite
dimensions~\cite{Shi:TSP:2014,Hong:MP:2017,Davis:MOR:2017,Giselsson:TAC:2017},
only few results exist for non-convex problems or nonlinear coupling. 
To this end and motivated by practical
experiences, the linear ADMM convergence is relaxed to R-linear
convergence in a next step. Under this assumption, it is shown that
there exists an upper, finite bound on the ADMM iterations,
such that the ADMM stopping criterion is satisfied and exponential
stability in closed loop as well as exponential decay of the
optimization error hold. This 
result also implies that the stopping criterion can be replaced by a
(sufficiently large) fixed number of iterations in each DMPC step to
circumvent the aforementioned global communication step for the
stopping criterion. 
Two example systems are used to illustrate the ADMM-based
DMPC scheme and to show the finiteness of the ADMM iterations as well
as the scalability of the approach.

The paper is outlined as follows: Section~\ref{sec:problem_statement}
introduces the coupled system dynamics and MPC formulation without terminal
constraints and summarizes stability results for the centralized MPC case. 
The decomposition and the DMPC scheme based on 
ADMM with a tailored stopping criterion is described in Section~\ref{sec:DMPC}. 
The stability analysis of the suboptimal DMPC scheme is carried out in
Section~\ref{sec:stability_results_dmpc}. Finally, 
Section~\ref{sec:simulation_results} presents simulation results for
two coupled systems to illustrate the ADMM-based DMPC scheme, before
conclusions are drawn in Section~\ref{sec:conclusions}.

Several notations and norms are used within the paper.
The standard norm used for a vector $ z \in \Rbb^n $ is the Euclidean norm
$\norm z := \norm z _2$ along with its induced matrix norm
$\norm A := \norm A_2$. The supremum norm $\norm{z}_{\infty}$ will be used occasionally.
In addition, the supremum norm of a function $ z(t) \in \Rbb^n$, $t \in [ 0, T ] $ is defined by 
$ \supnorm{z}  := \mysup_{t \in [0,T]} \norm{z(t)} _2$. 
A $\sigma$-neighborhood $ S^{\sigma} $ of a set $S \subset \Rbb^n $ is defined as
$ S^{\sigma} := \left\{ x \in \Rbb^n \,| \norm{x-\tilde{x}} \leq \sigma ~ \forall \tilde{x} \in S \right\}$.
Finally, system variables are underlined
  (e.g.\ $\ubar x$), in order to distinguish them from MPC-internal variables.

\section{System description and centralized MPC properties}
\label{sec:problem_statement}

The DMPC scheme in this paper is presented for nonlinear
continuous-time systems, where the subsystems are coupled via their
state variables. This section presents the problem statement along
with the considered MPC formulation and summarizes the stability
properties in a centralized MPC setting that serve as the basis for
the subsequent DMPC scheme in the following sections.

\subsection{Problem formulation}

Distributed systems consisting of coupled dynamical systems are
conveniently described by a directed graph $ \mathcal{G} = ( \Vcal, \mathcal{E} ) $, 
where the nodes $i\in \Vcal =\{1, \dots, N\}$ represent the dynamical
subsystems and the edges $\mathcal{E} \subseteq \Vcal \times \Vcal$
reflect the coupling between two subsystems.
The dynamics of subsystem $i\in\Vcal$ is described by
\begin{equation}\label{eq:dyn_subsys}
  \dot{\xsys}_i = f_i ( \xsys_i, \usys_i, [\xsys_j]_{j\in\Ngets{i}} )
\end{equation}
with the state $\xsys_i \in \Rbb^{n_i}$ and control $\usys_i \in
\Rbb^{m_i}$. 
The coupling to neighboring subsystems is given by the states $\xsys_j$ 
using the stacking notation
$[\xsys_j]_{j\in\Ngets{i}}\in\Rbb^{p_i}$ with 
$p_i := \sum_{ j \in \Ngets{i}} n_j$ and $p=\sum_{i\in\Vcal} p_i$, whereby 
$\Ngets{i} = \{j: (j,i) \in \mathcal{E}, j \not=i \}$. 
Vice versa, $\Nto{i} = \{j: (i,j) \in \mathcal{E}, j \not= i \}  $
represent the neighbors that subsystem $i\in\Vcal$ influences.
In view of a distributed control scheme, $\Ngets{i}$ represent the 
``sending'' neighbors, whereas $\Ngets{i}$ denote the ``receiving''
neighbors. 

\begin{figure}
	\centering
	{
		\normalsize
		\resizebox{0.3\textwidth}{!}{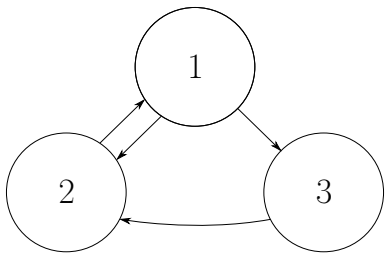}}
	\caption{Coupling structure for example system~\eqref{eq:DMPC_simple_example}.}
	\label{fig:couplStrucure}
\end{figure}

Figure~\ref{fig:couplStrucure} shows an example of a 
coupled system with the dynamics
\begin{align} \nonumber
\dot \xsys_1 &= f_1(\xsys_1,\usys_1,\xsys_2)
\\ \label{eq:DMPC_simple_example}
\dot \xsys_2 &= f_2(\xsys_2,\usys_2,\xsys_1,\xsys_3)
\\ \nonumber
\dot \xsys_3 &= f_3(\xsys_3,\usys_3,\xsys_1) \,.
\end{align}
The directed graph is defined by the knots $\Vcal=\{1,2,3\}$ and 
edges $\mathcal{E}=\{(1,2),(1,3),(2,1),(3,2)\}$.
The sending and receiving neighborhoods are 
$\Ngets{1}=\{2\}$, $\Ngets{2}=\{1,3\}$, $\Ngets{3}=\{1\}$ and 
$\Nto{1}=\{1,3\}$, $\Nto{2}=\{1\}$, $\Ngets{3}=\{2\}$, respectively.

The coupled subsystems~\eqref{eq:dyn_subsys} can be equivalently written in the
centralized form
\begin{equation}
\label{eq:dyn_sys}
\dot{\xsys}  = F (\xsys,\usys)
\end{equation}
with $F=[f_i]_{i\in\Vcal}$ and the stacked state and control vectors 
$\xsys = [ \xsys_i ]_{ i \in \Vcal }\in\Rbb^{n}$, 
$\usys = [ \usys_i ]_{ i \in \Vcal } \in \Rbb^{m}$ 
of dimension $ n := \sum_{i \in \Vcal} n_i$ and $m := \sum_{i \in \Vcal} m_i $,
respectively.
Without loss of generality, the control task consists in controlling
each subsystem \eqref{eq:dyn_subsys} to the origin, i.e.\
$0=f_i(0,0,0)$, $i\in\Vcal$ or equivalently $0=F(0,0)$ in the
centralized form~\eqref{eq:dyn_sys}. 

The MPC scheme to be considered is based on the online solution of the
optimal control problem (OCP) 
\begin{subequations}\label{eq:gl_ocp_in_gl_quant}
  \begin{align} \label{eq:gl_ocp_in_gl_quant_cost}
    \mymin\limits_{\umpc} \quad & J(\umpc,\xsys_k) := V( \xmpc(T) ) + \int_{0}^{T} { l ( \xmpc(\tau), \umpc(\tau) ) \dd \tau } \\
    \subjto  \quad & \dot{\xmpc}  = F (\xmpc,\umpc)\,, \quad \xmpc(0) = \xsys_k\\
    \label{eq:gl_ocp_in_gl_quant_constr}
    & \umpc \in U \text{,}
  \end{align}
where the cost functional \eqref{eq:gl_ocp_in_gl_quant_cost} with the
horizon length $T>0$ consists
of the integral and terminal cost functions for the single subsystems,
i.e.\
\begin{equation}
\label{eq:gl_ocp_in_gl_quant_cost_sum}
V(\xmpc) := \sum_{i\in \Vcal} V_i(\xmpc_i)\,, \quad
l(\xmpc,\umpc) := \sum_{i\in \Vcal} l_i(\xmpc_i,\umpc_i)
\end{equation}
\end{subequations}
with $l_i:\Rbb^n \times \Rbb^m \to \Rbb_0^+ $ and $V_i: \Rbb^n \to \Rbb_0^+$,
respectively. 
The set $U$ is the Cartesian product $\prod_{i\in\Vcal}U_i$ of the
compact and convex constraint sets $U_i$ for the controls $u_i$ of
the subsystems \eqref{eq:dyn_subsys}.
The initial value $\xsys(t_k) = \xsys_k = [ \xsys_{k,i} ]_{i \in \Vcal}$ is the
current (measured or observed) state of the system~\eqref{eq:dyn_subsys} at
sampling instant $t_k$.
The optimal solution of \eqref{eq:gl_ocp_in_gl_quant} is denoted by
$\xmpc^*(\tau;\xsys_k)$ and $\umpc^*(\tau;\xsys_k)$, $\tau\in[0,T]$
with the optimal (minimal) cost 
\begin{equation}
\label{eq:Jopt}
J^*(\xsys_k):=J(\umpc^*(\cdot;\xsys_k),\xsys_k)\,.
\end{equation}
As ususally done in MPC, the first part of the control trajectory
is applied to each subsystem~\eqref{eq:dyn_subsys} on the current sampling
interval
\begin{equation}
\label{eq:optcon:contrlaw}
\usys_i(t_k+\tau) = \umpc_i^*(\tau;\xsys_k) 
\,, \quad \tau\in[0,\dt)\,, \quad i\in\Vcal
\end{equation}
with the sampling time $\Delta t>0$.
From the centralized viewpoint and by the principle of optimality, the
optimal control can be interpreted as a nonlinear control law of the form 
\begin{equation}
\label{eq:optcontrlaw_centralized}
\umpc^*(\tau;\xsys_k) =:\kappa^*(\xmpc^*(\tau;\xsys_k);\xsys_k) \,,
\quad
\tau\in[0,\Delta t)\,.
\end{equation}
The integral and terminal cost functions in~\eqref{eq:gl_ocp_in_gl_quant_cost_sum}
are typically designed in quadratic form. In particular, it is assumed that there
exist constants $m_l,M_l>0$ and $m_V,M_V\ge 0$ such that the
centralized cost functions satisfy
\begin{equation}
  \label{eq:bounds_cost}\arraycolsep1mm
  \begin{array}{rcl}
    m_l  (\norm x ^2 + \norm u ^2) \leq & l(x,u) &\leq M_l ( \norm x ^2 + \norm u ^2 ) \\
    m_V \norm x ^2 \leq & V(x) &\leq M_V \norm x ^2 \text{.}
  \end{array}
\end{equation}
The dynamics function \eqref{eq:dyn_subsys}, resp.\
\eqref{eq:dyn_sys}, and the cost functions
\eqref{eq:gl_ocp_in_gl_quant_cost_sum} are supposed to be  
continuously differentiable. Moreover, throughout the paper, it is
assumed that every bounded control
trajectory $u(t)\in U$, $t\in[0,T]$ yields a bounded state response 
$x(t)$, $t\in[0,T]$ and that
an optimal solution of \eqref{eq:gl_ocp_in_gl_quant} exists.

\subsection{Centralized MPC}
\label{sec:MPC_centralized}

The MPC scheme is based on an OCP
formulation~\eqref{eq:gl_ocp_in_gl_quant} without terminal
constraints, which is motivated from the practical viewpoint to
reduce the computational burden in an actual implementation. This
section summarizes stability results from the literature on MPC without
terminal constraints~\cite{Limon:TAC:2006,Graichen:TAC:2010,Graichen:INTECH:2012}.
These results also serve as the basis for investigating the stability
properties of the DMPC scheme developed in the next sections.

Several assumptions are necessary at this point.
A common assumption in MPC is that the terminal cost $V(x)$ is
designed as a control Lyapunov function (CLF) detailed as follows.

\begin{assumption}
  \label{as:clf_on_omega}
  There exists a feedback law 
  $ u = \kappa_V (x) \in U $ 
  and a non-empty compact set 
  $ \Omega_{\beta}= \left\{ x \in \Rbb^n: V(x) \leq \beta \right\}$ 
  such that 
  $ \dot{V} ( x, \kappa_V(x) ) +  l ( x, \kappa_V(x) ) \leq 0$ $\forall x \in \Omega_{\beta}$ 
  with $ \dot{V} := (\partial V / \partial x ) F $.
\end{assumption}

Provided that the linearization of the system~\eqref{eq:dyn_sys} about the origin is
stabilizable, the terminal cost $V(x)$ and a linear feedback law
$\kappa_V(x)$ can be computed by solving a Riccati or Lyapunov
equation, which stabilizes the nonlinear system on a (possibly small)
set $\Omega_\beta$~\cite{Michalska:TAC:1993,Parisini:AUT:1995,Chen:AUT:1998}.

\begin{lemma} \label{lem:MPC:xT}
  Suppose that Assumption \ref{as:clf_on_omega} holds and consider the set 
  \begin{equation}
    \label{eq:levelsetalpha}
    \Gamma_\alpha = \left\{x \in \Rbb^n \,|\, J^*(x) \leq \alpha
    \right\}\,, \quad
    \alpha = \beta \left( 1 + \frac{m_l}{M_V}T \right)\,.
  \end{equation}
  Then, $\xmpc^*(T;\xsys_k)\in\Omega_\beta$ and $\Omega_\beta\subseteq\Gamma_\alpha$.
\end{lemma}

Lemma~\ref{lem:MPC:xT} states that the end point of the optimal
trajectory $\xmpc^*(T;\xsys_k)$ automatically lies inside the region
$\Omega_\beta$ of the CLF terminal cost. Moreover, $\Gamma_\alpha$
contains the region $\Omega_\beta$ and can be enlarged by increasing
$T$ as indicated by \eqref{eq:levelsetalpha}.
The proof of Lemma~\ref{lem:MPC:xT} can be found in
\cite{Graichen:INTECH:2012}.

This result eventually leads to the following
Theorem~\ref{theo:exp_stab_centr_MPC} concerning the stability 
of the centralized MPC scheme over the domain of attraction
$\Gamma_\alpha$.

\begin{theorem}  \label{theo:exp_stab_centr_MPC}
  Suppose that Assumption~\ref{as:clf_on_omega} holds. Then, 
  \begin{equation}
    \label{eq:cost_decrease_centr_MPC_terminal_cost}
    J^*(\xmpc^*(\dt,\xsys_k)) \le J^*(\xsys_k) - \int_0^{\dt}
    l(\xmpc^*(\tau;\xsys_k),\umpc^*(\tau;\xsys_k)) \dd\tau \quad
    \forall\xsys_k \in \Gamma_\alpha
  \end{equation}
  and the centralized MPC scheme 
  stabilizes the origin in the sense of 
  $\mylim_{t\to\infty}\|\xsys(t)\|=0$ for all
  $\xsys(0)=\xsys_0 \in \Gamma_\alpha$.
\end{theorem}

Theorem~\ref{theo:exp_stab_centr_MPC} (and the corresponding proof) is
a simplified version of the one given in \cite{Graichen:TAC:2010}. The
stability in the sense of $\mylim_{t\to\infty}\|\xsys(t)\|=0$ follows
from standard arguments in MPC using Barbalat's lemma. Stronger
results such as exponential stability require additional continuity
assumptions that will be employed later on in the paper.

\section{Decomposition and ADMM scheme}
\label{sec:DMPC}

The dynamically coupled OCP~\eqref{eq:gl_ocp_in_gl_quant} of the
centralized MPC scheme can be separated in terms of the single
subsystems using dual decomposition in connection with an augmented
Lagrangian formulation~\cite{Bertsekas:book:1997}. 
This is achieved by introducing local copies of the coupling variables
along with further coordination variables to allow for a decomposable
form of the augmented Lagrangian and a distributed solution based on ADMM.

\subsection{Decomposition}

The coupled subsystems~\eqref{eq:dyn_subsys} are decomposed by
introducing local copies of the coupling variables
$[\xsys_j]_{j\in\Ngets{i}}\in\Rbb^{p_i}$ that render the subsystem
dynamics independent of its neighbors $j\in\Ngets{i}$, i.e.\
\begin{equation}\label{eq:dyn_subsys_decoupled}
  \dot{\xmpc}_i = f_i \left( \xmpc_i, \umpc_i, v_i \right) \,, \quad
  \xmpc_i(0) = \xsys_{k,i} \,, \quad
  i\in\Vcal \,.
\end{equation}
The local copies are defined as 
\begin{equation}
v_i=[v_{ji}]_{j\in\Ngets{i}}\,, \quad i\in\Vcal
\end{equation}
where $v_{ji}$ represents the local copy of state $\xmpc_j$, $j\in\Ngets{i}$ 
for subsystem $i\in\Vcal$.
In order to use a separable augmented Lagrangian formulation for the
distributed ADMM algorithm, further ``coordination'' variables $z_i$
along with the consistency constraints  
\begin{subequations}
\label{eq:consistencyconstr}
\begin{alignat}{2}
\label{eq:consistencyconstr_1}
0 &= z_i - x_i \,, & \quad & i\in\Vcal
\\ \label{eq:consistencyconstr_2}
0 &= z_j - v_{ji} \,, & \quad & j\in \Ngets{i} \,, \quad i\in\Vcal \,.
\end{alignat}
\end{subequations}
are introduced.\,\footnote{
  Note that this approach differs from the ADMM approach in
  \cite{Farokhi:Springer:2014} that is based on a non-augmented Lagrangian
  formulation without quadratic penalty, for which the local copies
  $v_i$ of the coupling variables are sufficient to separate the
  Lagrangian.  
}

Figure~\ref{fig:couplVars} illustrates the decomposition for the
previous example system in \eqref{eq:DMPC_simple_example} and
Figure~\ref{fig:couplStrucure}. The local copies of the subsystems are
given by 
\begin{equation}
v_1 = v_{21}\,,  \quad 
v_2 = \begin{bmatrix} v_{12} \\ v_{32} \end{bmatrix} , 
\quad 
v_3 =v_{13} 
\end{equation}
and the consistency constraints \eqref{eq:consistencyconstr_1} and
\eqref{eq:consistencyconstr_2} read 
\begin{equation}
\left.\begin{aligned}
0 &= z_1-x_1 \\ 0 &= z_2-x_2 \\ 0 &= z_3-x_3
\end{aligned}\right\}
\quad \text{and} \quad
\left\{\begin{aligned}
0 &= z_1-v_{12} \\ 0 &= z_1-v_{13} \\ 0 &= z_2-v_{21} \\ 0 &= z_3-v_{32} 
\end{aligned}\right.
\end{equation}
accordingly.

\begin{figure}[t]
	\centering
	{
		\normalsize
		\resizebox{0.4\textwidth}{!}{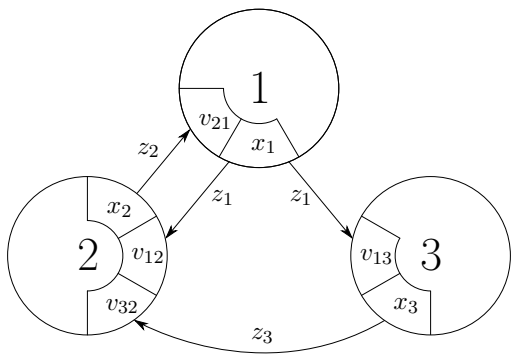}}
	\caption{Coupling variables for the decomposition of example
          system~\eqref{eq:DMPC_simple_example}.}
	\label{fig:couplVars}
\end{figure}

To ease notations, we introduce the stacking notation and multipliers
\begin{equation}
\label{eq:notation_w_mu}
w = \begin{bmatrix} u_i \\ v_i\end{bmatrix}_{i\in\Vcal}\,, \quad
z = [z_i]_{i\in\Vcal}\,, \quad \mu = [\mu_i]_{i\in\Vcal} 
\quad\text{with}\quad \mu_i
= \begin{bmatrix} 
\hphantom{[}\mu_{ii}\hphantom{]_{j\in\Ngets{i}}} \\ [\mu_{ji}]_{j\in\Ngets{i}}
\end{bmatrix},
\end{equation}
and define the augmented Lagrangian 
\begin{multline}
\label{eq:augmLagr}
J_\rho(w,z,\mu;\xsys_k) := J(\umpc;\xsys_k) +
\sum_{i\in\Vcal} \int_0^T 
\mu_{ii}^\TT (z_i-\xmpc_i) + 
  \frac{\rho}{2} ||z_i-\xmpc_i||^2
\\
+ \sum_{j\in\Ngets{i}} \left(\mu_{ji}^\TT(z_j-v_{ji}) + \frac{\rho}{2} ||z_j-v_{ji}||^2\right)
\dd\tau
\end{multline}
subject to the dynamics~\eqref{eq:dyn_subsys_decoupled}.
The consistency constraints~\eqref{eq:consistencyconstr} are adjoined
to \eqref{eq:augmLagr} by means of the multipliers defined in
\eqref{eq:notation_w_mu}, where
$\mu_{ii}$ and $\mu_{ij}$ are associated with
\eqref{eq:consistencyconstr_1} and \eqref{eq:consistencyconstr_2}, respectively. 
The consistency constraints are additionally penalized in
\eqref{eq:augmLagr} with parameter $\rho>0$, as typically done in
augmented Lagrangian 
formulations. 

Provided that strong duality holds, the optimal pairing
$(w^*,z^*,\mu^*)=(w^*,z^*,\mu^*)(\tau;\xsys_k)$ satisfies the saddle point condition
\begin{equation}
\label{eq:saddlepointcond}
J_\rho(w^*,z^*,\mu;\xsys_k) \le 
J_\rho(w^*,z^*,\mu^*;\xsys_k) \le 
J_\rho(w,z,\mu^*;\xsys_k) \,,
\end{equation}
which implies that the optimal value $J_\rho(w^*,z^*,\mu^*;\xsys_k)$ 
is minimized w.r.t.\ the primal variables $(w,z)$ and maximized in
terms of the multipliers $\mu$ respresenting the dual variables. 
Instead of solving OCP~\eqref{eq:gl_ocp_in_gl_quant} directly, we can
employ an ADMM algorithm where each iteration $q$ is given by%
\begin{subequations}
  \label{eq:ADMM_centralized}
\begin{alignat}{2}
  \label{eq:ADMM_centralized_1}
  w^{q} &= \myarg \mymin_{w\in\Wcal} J_\rho(w,z^{q-1},\mu^{q-1};\xsys_k)
  \\   \label{eq:ADMM_centralized_2}
  z^{q} &= \myarg \mymin_{z} J_\rho(w^{q},z,\mu^{q-1};\xsys_k)
  \\   \label{eq:ADMM_centralized_3}
  \mu^{q} &= \mu^{q-1} + \rho \begin{bmatrix}
    \hphantom{[} z_i^{q}-x_i^{q} \hphantom{]_{j\in\Ngets{i}}} \\
    \left[z_j^{q}-v_{ji}^{q}\right]_{j\in\Ngets{i}}
  \end{bmatrix}_{i\in\Vcal}
\end{alignat}
\end{subequations}
with $\Wcal=\{(u,v) \in L_{m+p}^\infty(0,T)\,|\, u(t)\in U~\forall t\in[0,T]\}$,
consisting of the minimization steps \eqref{eq:ADMM_centralized_1}, 
\eqref{eq:ADMM_centralized_2} and the steepest ascent update
\eqref{eq:ADMM_centralized_3} for the multipliers. Note that the
minimization problem \eqref{eq:ADMM_centralized_1} uses the original
control variables as well as the local copies as optimization
variables.  

An assumption and discussion on the convergence of the ADMM algorithm
can be found in Section~\ref{sec:stability_results_dmpc}. The next
subsection exploits the separability of the augmented Lagrangian
\eqref{eq:augmLagr} to derive a distributed MPC algorithm for the
dynamically coupled system~\eqref{eq:dyn_subsys}. In particular, a
suitable stopping criterion for the ADMM
iterates~\eqref{eq:ADMM_centralized} is necessary, which will be in
the focus of the remainder of the paper.

\subsection{Distributed solution by ADMM}
\label{sec:ADMM_distr}

The augmented Lagrangian \eqref{eq:augmLagr} can be separated in the
single subsystems $i\in\Vcal$, which allows to solve the ADMM
iterations in a distributed manner. Moreover, the minimization
\eqref{eq:ADMM_centralized_2} can be solved analytically,
as $z$ is not involved in the system dynamics~\eqref{eq:dyn_subsys_decoupled}.
The following lines summarize the distributed ADMM algorithm for
each local MPC agent associated to subsystem $i\in\Vcal$ that is
executed in each MPC step with the current system state
$\xsys_k=[\xsys_{k,i}]_{i\in\Vcal}=\xsys(t_k)$. The dependency of the
trajectories 
on $\tau\in[0,T]$ and $\xsys_k$ 
is omitted in the
algorithm for the sake of readability.\newline
\vspace{2pt}
\hrule
\textbf{1) Local initialization}
  \begin{itemize}
  \item Choose $\rho>0$ (penalty parameter) and $d>0$ (stopping criterion).
  \item Initialize $(z_i^0,\mu_i^0)$ and set $q=1$.
    \item Receive $z_j^0$ from neighbors $j\in\Ngets{i}$ and 
      $\mu_{ij}^0$ from neighbors $j\in\Nto{i}$.
  \end{itemize}

\textbf{2) Local minimization}

  \begin{itemize}
    \item Compute 
      $(w_i^q,\xmpc_i^q)$ with $w_i=[\umpc_i^\TT,v_i^\TT]^\TT$ by solving
      \begin{subequations}
        \label{eq:cost:MPClocal}
      \begin{alignat}{3} \nonumber
        & \!\mymin_{w_i}\quad  && V_i(\xmpc(T))+\int_0^T 
        l_i(\xmpc_i,\umpc_i) + (\mu_i^{q-1})^\TT(z_i^{q-1}-\xmpc_i) +
        \frac{\rho}{2}||z_i^{q-1}-\xmpc_i||^2
        \\ \label{eq:cost:MPClocal:cost}
        &&& \qquad + \sum_{j\in\Ngets{i}}\left( (\mu_{ji}^{q-1})^\TT(z_j^{q-1}\!-\!v_{ji})+
        \frac{\rho}{2}||z_j^{q-1}\!-\!v_{ji}||^2\right)\!\dd\tau
      \\ \label{eq:cost:MPClocal:dyn}
      & \,\text{s.t.} && \dot{\xmpc}_i = f_i(\xmpc_i,\umpc_i,v_i)\,, \quad
      \xmpc_i(0) = \xsys_{k,i} 
      \\ &&&
      u_i(\tau) \in U_i\,, \quad \tau\in[0,T] \,.
    \end{alignat}
    \end{subequations}
  \item Receive $v_{ij}^q$ from neighbors $j\in\Nto{i}$.
  \item Compute $z_i^q$ with
    \begin{equation}
      \label{eq:ADMM:update_z}
      z_i^q = \frac{1}{|\Nto{i}|+1} \left( 
        \xmpc_i^q + \sum_{j\in\Nto{i}} v_{ij}^q - \frac{1}{\rho} 
        \sum_{j\in\Nto{i}\cup\{i\}} \mu_{ij}^{q-1}
      \right).
    \end{equation}
  \item Receive $z_j^q$ from neighbors $j\in\Ngets{i}$.
  \end{itemize}
\textbf{3) Local multiplier update}
    \begin{itemize}
    \item Compute $\mu_i^q$ with
      \begin{equation}
        \label{eq:ADMM:update_mu}
        \mu_i^q = \mu_i^{q-1} + \rho \begin{bmatrix} 
          \hphantom{[} z_i^q-\xmpc_i^q \hphantom{]_{j\in\Nto{i}}}\\
          \left[z_j^q - v_{ji}^q \right]_{j\in\Ngets{i}}\end{bmatrix}\,.
      \end{equation}

    \item Compute local copies of $\mu_{ij}^q$ with
      \begin{equation}
        \label{eq:ADMM:update_mu_copy}
        \mu_{ij}^q = \mu_{ij}^{q-1} + \rho (z_i^q-v_{ij}^q) \,,
        \quad j\in\Nto{i}\,.
      \end{equation}
    \end{itemize}
\textbf{4) Stopping criterion}
      \begin{itemize}
      \item Quit, if 
        \begin{equation}
          \label{eq:ADMM_convcrit}
          \begin{Vmatrix} 
            z_i^q - z_i^{q-1} \\ 
            \mu_i^{q} - \mu_i^{q-1} 
          \end{Vmatrix}_{L^\infty} 
          \!\!\le d\,||\xsys_{k,i}||
        \end{equation}
        is satisfied in iteration $q_k:=q$ and for all $i\in\Vcal$.
      \item Otherwise, increment $q$ and return to 2).
      \end{itemize}

The minimization problems~\eqref{eq:cost:MPClocal} are local MPC
problems directly following from the splitting of the augmented
Lagrangian \eqref{eq:augmLagr} subject to the decoupled
dynamics~\eqref{eq:dyn_subsys_decoupled}. The explicit computations of
the local variables $z_i^q$ according to \eqref{eq:ADMM:update_z}
directly follows from the analytic solution of the minimization
problem \eqref{eq:ADMM_centralized_2}, as mentioned before.
The multiplier update in Step 3) computes $\mu_i^q$ 
as defined in~\eqref{eq:notation_w_mu}. Moreover, the additional step 
\eqref{eq:ADMM:update_mu_copy} generates local copies of the
multipliers $\mu_{ij}^q$ associated with the neighboring MPC agents $j\in\Nto{i}$
in order to avoid an additional communication step. 

An important part of the algorithm concerns the stopping criterion in
Step~4) that terminates the ADMM algorithm at a certain iteration number
$q_k$ in the current MPC step $k$. The criterion \eqref{eq:ADMM_convcrit} 
evaluates the progress of $(z_i^q,\mu_i^q)$ between
two iterations. In particular, the norm of $\mu_i^{q}-\mu_i^{q-1}$ is a
direct measure for the residual of the consistency
constraints~\eqref{eq:consistencyconstr}, which becomes obvious in
view of the multiplier update~\eqref{eq:ADMM:update_mu}. The
stopping criterion is formulated in terms of the norm of the current local
system state $\xsys_{k,i}$ with the constant $d>0$. This leads to a
contraction of the stopping criterion during the stabilization of the
overall system to the origin.
The stopping criterion \eqref{eq:ADMM_convcrit} is locally evaluated,
but has to be satisfied by all MPC agents $i\in\Vcal$. 
This requires one global communication step, either by broadcasting
the convergence flag of each MPC agent to the whole network
or alternatively by defining a master agent that collects all
convergence flags and in turn communicates the termination of the
algorithm to all agents if \eqref{eq:ADMM_convcrit} is satisfied.
Note that the evaluation of the stopping criterion involves the only global
communication step of the algorithm, as the updates during one ADMM
iteration only require neighbor-to-neighbor communication.\,\footnote{
  As it will
  turn out in Section~\ref{sec:relaxed_assumptions}, an alternative to
  the stopping criterion in Step 4) is to use a sufficiently large
  number of iterations without explicitely evaluating
  \eqref{eq:ADMM_convcrit}. 
}

In the first MPC step $k=0$, the ADMM algorithm is locally initialized
in Step 1) with appropriate trajectories $(z_i^0,\mu_i^0)$.
The subsequent MPC steps $k\ge 1$ rely on the solutions from the
previous run, i.e.\  $(z_i^0,\mu_i^0)=(z_i^0,\mu_i^0)(\tau;\xsys_k)$ are
re-initialized in a warm-start manner by
\begin{align} 
\label{eq:multipl_initialization}
  \mu^0(\tau; \xsys_{k+1}) = \mu^{q_k}(\tau; \xsys_k)\,, \quad
  z^0(\tau;\xsys_{k+1}) = z^{q_k}(\tau; \xsys_k) \,, \quad
  \tau\in[0,T]\,.
\end{align}
When the stopping criterion \eqref{eq:ADMM_convcrit} is satisfied, 
the control trajectories $\umpc^{q_k}=\umpc^{q_k}(\tau;\xsys_k)$ from
the last iteration $q_k$ are 
used as local 
controls for the actual subsystems~\eqref{eq:dyn_subsys}, i.e.\
\begin{equation}
\usys_i(t_k+\tau) = \umpc_i^{q_k}(\tau;\xsys_k) \,, \quad
\tau\in[0,\dt)\,, \quad i\in\Vcal\,.
\end{equation}
Similar to \eqref{eq:optcontrlaw_centralized}, the principle of
optimality allows one to express the computed controls 
$\umpc^{q_k}=[\umpc_i^{q_k}]_{i\in\Vcal}$ as a nonlinear control law of the form
\begin{equation}
\label{eq:optcontrlaw_distributed}
\umpc^{q_k}(\tau;\xsys_k) =: \kappa(\xmpc^{q_k}(\tau;\xsys_k);z^{q_k-1},\mu^{q_k-1},\xsys_k) 
\end{equation}
with $\xmpc^{q_k}=[\xmpc_i^{q_k}]_{i\in\Vcal}$.
Note, however, that in contrast to the centralized MPC case
\eqref{eq:optcontrlaw_centralized}, $\kappa$ is 
parametrized by the trajectories of the previous iteration 
$z^{q_k-1}(\cdot;\xsys_k)$ and $\mu^{q_k-1}(\cdot;\xsys_k)$.

\section{Distributed MPC stability} \label{sec:stability_results_dmpc}

The distributed MPC scheme presented in the last section is
suboptimal in the sense that the stopping criterion leads to a 
premature stop of the iterations. As a consequence, the
ADMM-based solution is not consistent with the centralized
MPC solution in Section~\ref{sec:MPC_centralized} and
requires to take a closer look at the distributed MPC scheme and the
effects of the stopping criterion.
This section therefore derives stability results for the DMPC scheme
under linear convergence assumptions on the ADMM algorithm that are
subsequently relaxed in a second step.

\subsection{Preliminaries}
\label{sec:DMPC_preliminaries}

The stopping criterion~\eqref{eq:ADMM_convcrit} has the advantage to limit
the ADMM iterations and the corresponding communication effort
in each MPC step. On the downside, the premature exit implies that the
computed trajectories $\umpc^{q_k}(\cdot;\xsys_k)$ do not correspond to
the optimal ones $\umpc^*(\cdot;\xsys_k)$ from the centralized
OCP~\eqref{eq:gl_ocp_in_gl_quant}. 
In particular, the consistency
constraints~\eqref{eq:consistencyconstr} are not exactly satisfied,
leading to different state trajectories of the local MPC
predictions, the actual realizations and the optimal ones of the
centralized MPC solution.
In detail, we have to distinguish between the following state
trajectories:
\begin{itemize}
\item the \emph{individual} predicted state trajectories
  $\xmpc^{q_k}(\cdot;\xsys_k)=[\xmpc_i^{q_k}(\cdot;\xsys_{k,i})]_{i\in\Vcal}$ as part of the
  solution of the local OCPs~\eqref{eq:cost:MPClocal}, i.e.\
  \begin{equation}
    \label{eq:traj_ind}
    \dot\xmpc^{q_k}(\tau;\xsys_k) = f(\xmpc^{q_k}(\tau;\xsys_k), 
    \umpc^{q_k}(\tau;\xsys_k), v^{q_k}(\tau;\xsys_k)) \,, \quad 
    \xmpc^{q_k}(0;\xsys_k) = \xsys_k \,,
  \end{equation}
  where $f = [f_i]_{i\in\Vcal}$ and $v = [v_i]_{i\in\Vcal}$
  stack the dynamics functions $f_i$ and local copies $v_i$ in
  \eqref{eq:cost:MPClocal:dyn},
  
\item the \emph{actual} predicted state trajectory 
  $\xmpc_F^{q_k}(\cdot;\xsys_k)$ 
  following from applying the controls $\umpc^{q_k}(\cdot;\xsys_k)$ to the
  system~\eqref{eq:dyn_sys}, i.e.\
  \begin{equation}
    \label{eq:traj_act}
    \dot\xmpc_F^{q_k}(\tau;\xsys_k) =
    F(\xmpc_F^{q_k}(\tau;\xsys_k),
    \umpc^{q_k}(\tau;\xsys_k)) \,, \quad \xmpc_F^{q_k}(0;\xsys_k) =
    \xsys_k \,.
  \end{equation}

\item the \emph{optimal} predicted state trajectory 
  $\xmpc^*(\cdot;\xsys_k)=[\xmpc_i^*(\cdot;\xsys_{k,i})]_{i\in\Vcal}$
  following from solving the centralized OCP~\eqref{eq:gl_ocp_in_gl_quant}, i.e.\
  \begin{equation}
    \label{eq:traj_opt}
    \dot\xmpc^*(\tau;\xsys_k) = 
    F(\xmpc^*(\tau;\xsys_k),\umpc^*(\tau;\xsys_k)) \,, \quad \xmpc^*(0;\xsys_k) =
    \xsys_k \,.
  \end{equation}
\end{itemize}
In general, the single trajectories will differ from each other due to
the premature stopping criterion~\eqref{eq:ADMM_convcrit}.
In the receding horizon fashion of MPC this implies that the 
system state $\xsys_{k+1}=\xsys(t_{k+1})$ at the next sampling instant
will lie on the actual state trajectory $\xsys_{k+1}=x_F^{q_k}(\dt;\xsys_k)$
and not on the individual trajectory $\xmpc^{q_k}(\dt;\xsys_k)$ or the
optimal one $\xmpc^*(\dt;\xsys_k)$. This discrepancy is captured in
the error between the actual and optimal trajectory
\begin{equation}
  \label{eq:actualerror}
  \Delta x_F^{q_k}(\tau;\xsys_k) := \xmpc_F^{q_k}(\tau;\xsys_k) -
  \xmpc^*(\tau;\xsys_k) \,, \quad \tau\in[0,T]
\end{equation} 
that can be interpreted as the optimization error or suboptimality
measure in each MPC step $k$. The error \eqref{eq:actualerror} will be
of importance for the stability analysis in
Section~\ref{sec:stability_subsection}. To this end, several
assumptions are necessary to proceed.

\begin{assumption} \label{ass:lipschitz_DMPC}
The optimal cost \eqref{eq:Jopt} is twice continuously
differentiable and the optimal feedback laws
\eqref{eq:optcontrlaw_centralized} and
\eqref{eq:optcontrlaw_distributed} 
are locally Lipschitz in their arguments.   
\end{assumption}

The continuity properties in Assumption~\ref{ass:lipschitz_DMPC} are
necessary to derive certain bounds during the stability analysis.
Another assumption concerns the existence of optimal multipliers for
the augmented Lagrangian formulation~\eqref{eq:augmLagr}.

\begin{assumption} \label{as:unique_bounded_multipl}
  There exist unique bounded multipliers $\mu^*(\tau;\xsys_k)$,
  $\tau\in[0,T]$ satisfying the saddle point condition \eqref{eq:saddlepointcond}.
  Moreover, $\mu^*(\cdot;\xsys_k)$ and $z^*(\cdot;\xsys_k)$ are
  locally Lipschitz for all $\xsys_k\in\Gamma_\alpha$. 
\end{assumption}

For the ADMM algorithm~\eqref{eq:ADMM_centralized} and its distributed
solution in Section~\ref{sec:ADMM_distr}, we have to assume
boundedness as well as convergence for the iterated solutions. To this
end, denote $\Mcal_\alpha:=\{\mu^*(\cdot;\xsys_k) \,|\,\xsys_k\in\Gamma_\alpha\}$ and
$\Zcal_\alpha:=\{z^*(\cdot;\xsys_k) \,|\,\xsys_k\in\Gamma_\alpha\}$
the compact and non-empty sets of optimal multipliers $\mu^*$ and
variables $z^*$ 
and let $\Mcal_\alpha^{\sigma_\mu}$ and $\Zcal_\alpha^{\sigma_z}$ be
corresponding $\sigma$-neighborhoods with $\sigma_\mu,\sigma_z>0$. 

\begin{assumption}
  \label{as:ADMM_boundedness}
  The ADMM algorithm \eqref{eq:ADMM_centralized} generates bounded
  trajectories in the sense that $\exists\sigma_\mu,\sigma_z>0$ such
  that $\mu^q(\cdot;\xsys_k) \in \Mcal_\alpha^{\sigma_\mu}$, 
  $z^q(\cdot;\xsys_k) \in \Zcal_\alpha^{\sigma_z}$ 
  for a given initialization ($q=0$) and the subsequent iterations
  ($q=1,2,\ldots$).
\end{assumption}

\begin{assumption}
  \label{as:linear_convergence}
  The ADMM iterates are linearly convergent in the sense
  \begin{equation}
    \label{eq:linear_convergence}
    \begin{Vmatrix}
      z^{q}(\cdot;\xsys_k) - z^*(\cdot;\xsys_k) \\ 
      \mu^{q}(\cdot;\xsys_k) - \mu^*(\cdot;\xsys_k)  
    \end{Vmatrix}_{L^\infty}
    \!\! \le C
     \begin{Vmatrix}
       z^{q-1}(\cdot;\xsys_k) - z^*(\cdot;\xsys_k) \\ 
       \mu^{q-1}(\cdot;\xsys_k) - \mu^*(\cdot;\xsys_k) 
     \end{Vmatrix}_{L^\infty}
  \end{equation}
  for some $C\in[0,1)$.
\end{assumption}

The linear convergence assumption is certainly the strongest
assumption for proving stability of the suboptimal ADMM-based DMPC
scheme. The investigation of 
convergence for ADMM algorithms has been in the focus of research over
the last years, mostly assuming finite-dimensional and convex
optimization problems. For instance, linear ADMM convergence rates either 
for primal and/or dual variables are shown
in~\cite{Han:SIAM:2013,Shi:TSP:2014,Deng2016,Hong:MP:2017}. 
For ADMM in an infinite-dimensional setting but with convexity
assumptions, 
\cite{Gabay:1983,Davis2015,Giselsson2015,Davis:MOR:2017,Giselsson:TAC:2017} prove
convergence by applying the Peaceman-Rachford or 
Douglas-Rachford method to the dual problem, which is equivalent to
using ADMM in the primal case.
However, non-convex problems or nonlinear couplings have only
been considered sporadically or for specific
problems~\cite{Attouch:MOR:2010,Benning:2016,Magnusson2016,Wang2016}.
In particular and to the best knowledge of the authors, there exist no
convergence results in the case of ADMM for state-coupled nonlinear
OCPs as considered in this paper, though the above-mentioned
references may lead to similar convergence results in the future.
Moreover, a relaxation of Assumption~\ref{as:linear_convergence} will
be considered in Section~\ref{sec:relaxed_assumptions}.

\subsection{Stability results}
\label{sec:stability_subsection}

The investigation of stability requires to look at the
error $\Delta x_F^{q_k}(\cdot;\xsys_k)$ defined in~\eqref{eq:actualerror}
between the suboptimal ADMM-based DMPC scheme and the centralized,
optimal MPC solution. The following lemma gives an 
intermediate result that bounds the error $\Delta x_F^{q_k}(\cdot;\xsys_k)$
in terms of the stopping criterion~\eqref{eq:ADMM_convcrit} for all
MPC agents $i\in\Vcal$.

\begin{lemma} \label{lem:err_to_stop_crit}
  Suppose that Assumptions
  \ref{as:clf_on_omega}-\ref{as:linear_convergence} hold.
  Then, there exists a constant $D>0$ such that the error
  \eqref{eq:actualerror} satisfies
  \begin{equation}
    \label{eq:err_to_stop_crit}
    \|\Delta\xmpc_F^{q_k}(\cdot;\xsys_k)\|_{L^\infty} \le D d \|\xsys_k\|
    \quad \forall\xsys_k\in\Gamma_\alpha\,.
  \end{equation}
\end{lemma}
\begin{proof}
  See Appendix~\ref{app:sec:proof_err_to_stop_crit}.
\end{proof}

In the optimal, centralized MPC case, the system state in the next
sampling step $t_{k+1}=t_k+\Delta t$ lies on the optimal state
trajectory~\eqref{eq:traj_opt} and the optimal cost
$J^*(\xmpc^*(\Delta t;\xsys_k))$ decreases according 
to~\eqref{eq:cost_decrease_centr_MPC_terminal_cost}.  
This is not the case for the DMPC scheme, since
the error $\Delta x_F^{q_k}(\cdot;\xsys_k)$ has a direct influence on the
behavior of the optimal cost from one step to the next. In particular,
the next system state is given by 
\begin{equation}
\label{eq:actual_state_dt}
x_F^{q_k}(\Delta t;\xsys_k) = \xmpc^*(\Delta t;\xsys_k) + 
\Delta x_F^{q_k}(\cdot;\xsys_k) \,.
\end{equation} 
The next lemma bridges the gap between the optimal cost decay given 
in \eqref{eq:cost_decrease_centr_MPC_terminal_cost} and 
the distributed, suboptimal ADMM solution.

\begin{lemma} \label{lem:cost_subopt_sol}
  Suppose that Assumption~\ref{ass:lipschitz_DMPC} holds. 
  Then, there exists a set $\Gamma_{\hat\alpha}
  \subset\Gamma_\alpha$ with $\hat\alpha=\hat\alpha(d,\dt)$
  such that $\xmpc_F^{q_k}(\dt;\xsys_k)\in\Gamma_\alpha$ for all
  $\xsys_k\in\Gamma_{\hat\alpha}$.  
  Moreover, there exist constants $ 0 < a \leq 1 $ and $ b,c > 0 $ 
  such that $\forall\xsys_k \in \Gamma_{\hat\alpha}$, 
  \begin{equation}\label{eq:cost_subopt_sol}
    J^*(\xmpc_F^{q_k}(\dt;\xsys_k)) \le 
    (1-a) J^*(\xsys_k) + b \sqrt{J^*(\xsys_k)}
    \|\Delta\xmpc_F^{q_k}(\cdot;\xsys_k)\|_{L^\infty}
    + c \|\Delta\xmpc_F^{q_k}(\cdot;\xsys_k)\|_{L^\infty}^2 \,.
  \end{equation}
\end{lemma}

\begin{proof}
To prove the first statement of the lemma, consider the bound on the
actual state trajectory at sampling time $\Delta t$ as given
in~\eqref{eq:actual_state_dt} 
\begin{align} \nonumber
\|\xmpc_F^{q_k}(\Delta t;\xsys_k)\| &\le 
\|\xmpc^*(\Delta t;\xsys_k)\| + 
\| \Delta x_F^{q_k}(\cdot;\xsys_k) \|
\\ &\le \label{eq:proof_subopt_sol_1}
\left( e^{\hat L \Delta t} + Dd \right) \|\xsys_k\|\,,
\end{align}
where the second line follows from \eqref{eq:err_to_stop_crit} in
Lemma~\ref{lem:err_to_stop_crit} 
and the bound \eqref{eq:app:upper_bound_xopt} in
Appendix~\ref{sec:helpfulbounds}. 
Moreover, note that for any $x\in\Gamma_{\alpha}$ and any 
$\alpha>0$ we have $\|x\|\le \sqrt{\frac{\alpha}{m_J}}$. Vice
versa, $\|x\| \le \sqrt{\frac{\alpha}{M_J}}$ implies
$x\in\Gamma_\alpha$.
Both statements follow from the set definition \eqref{eq:levelsetalpha}  
together with the quadratic bounds
\eqref{eq:app_J_upper_bound} and \eqref{eq:app_J_lower_bound}.
Hence, for $\xsys_k \in\Gamma_{\hat\alpha}$ with 
$\hat\alpha:=\frac{m_J\alpha}{M_J(e^{\hat L\Delta t}+Dd)^2}$
we have $\|\xsys_k\|\le\frac{1}{e^{\hat L\Delta t}+Dd}\sqrt{\frac{\alpha}{M_J}}$.
The bound \eqref{eq:proof_subopt_sol_1} then becomes
$\|\xmpc_F^{q}(\Delta t;\xsys_k)\|\le\sqrt{\frac{\alpha}{M_J}}$,
which implies $\xmpc_F^{q}(\Delta t;\xsys_k) \in\Gamma_\alpha$ as
reasoned above.

To ease notations in the following lines, we define
$\xmpc_{k+1}=\xmpc_F^{q_k}(\Delta t;\xsys_k)$ and  
$\xmpc_{k+1}^*=\xmpc^*(\Delta t;\xsys_k)$ that both are located
inside $\Gamma_\alpha$
and consider the following line integral along a linear path 
$\xmpc_{k+1}^* + s\, \Delta x_{k+1}$ with
$\Delta x_{k+1}=\Delta\xmpc_F^{q_k}(\Delta t;\xsys_k)$ and $s\in[0,1]$, 
which by the choice of $\hat{\alpha}$ likewise lies
completely within $\Gamma_\alpha$: 
\begin{alignat}{3} \nonumber
J^*(\xmpc_{k+1}) &= J^*(\xmpc_{k+1}^*) &&+ 
\int_0^1 
\nabla J^*( \xmpc_{k+1}^* + s\, \Delta\xmpc_{k+1})\, \Delta\xmpc_{k+1} \dd s
\\ \nonumber &=
J^*(\xmpc_{k+1}^*) &&+ \int_0^1 \Big[
\nabla J^*(x_{k+1}^*)
+ \int_0^s \nabla^2 J^*( \xmpc_{k+1}^* + s_2\, \Delta\xmpc_{k+1} )
\, \Delta\xmpc_{k+1} \dd s_2 \Big]
\Delta\xmpc_{k+1} \dd s
\\ 
& \label{eq:proof_subopt_sol_2}
\le J^*(\xmpc_{k+1}^*) &&+ B \|\xmpc_{k+1}^*\|
\,\|\Delta\xmpc_F^{q_k}(\cdot;\xsys_k)\|_{L^\infty}
+ \tfrac{1}{2} B \|\Delta\xmpc_F^{q_k}(\cdot;\xsys_k) \|_{L^\infty}^2\,.
\end{alignat}
Since $J^*(\cdot)$ is two times differentiable over the compact set
$\Gamma_\alpha$, there exists a constant $B>0$ such that 
$\|\nabla J^*(x)\|\le B\|x\|$ and $\|\nabla^2 J^*(x)\|\le B$ for all
$x\in\Gamma_\alpha$, which yields the last line in
\eqref{eq:proof_subopt_sol_2}. 
The first term in \eqref{eq:proof_subopt_sol_2} can be expressed as
$J^*(\xmpc_{k+1}^*)\le (1-a) J^*(\xsys_k)$ using
\eqref{eq:cost_decrease_centr_MPC_terminal_cost}
with \eqref{eq:app_intdt_lower_bound}. In addition,
\eqref{eq:app:upper_bound_xopt} and \eqref{eq:app_J_lower_bound}
give $\|\xmpc^*_{k+1}\|\le \frac{e^{\hat L \dt}}{\sqrt{m_J}}
\sqrt{J^*(\xsys_k)}$, which eventually yields
\eqref{eq:cost_subopt_sol} with $a$ given in \eqref{eq:app_intdt_lower_bound},
$b=\frac{e^{ \hat{L} \Delta t }}{ \sqrt{m_J} }B$, and
$ c = B/2 $.
\end{proof}

In contrast to the centralized MPC case, where the optimal
cost~\eqref{eq:cost_decrease_centr_MPC_terminal_cost} 
decreases from one MPC step to the next, the
relation~\eqref{eq:cost_subopt_sol} reveals the influence of the error 
$\|\Delta\xmpc_F^{q_k}(\cdot;\xsys_k)\|$ that opposes
the contraction term $(1-a)J^*(\xsys_k)$.
Moreover, Lemma~\ref{lem:cost_subopt_sol} restricts the original domain of
attraction $\Gamma_\alpha$ from the centralized MPC case to the
smaller set $\Gamma_{\hat\alpha}$. 
Based on Lemma~\ref{lem:err_to_stop_crit} and
\ref{lem:cost_subopt_sol}, the stability of the DMPC scheme is
shown in the following theorem.

\begin{theorem} \label{th:exp_stab_DMPC}
Suppose that Assumptions
\ref{as:clf_on_omega}-\ref{as:linear_convergence} hold.
Then, there exists a sufficiently small
constant $d>0$ in the ADMM stopping criterion
\eqref{eq:ADMM_convcrit}
such that the optimal cost \eqref{eq:Jopt} as well as
the error \eqref{eq:actualerror} decay exponentially for all 
$\xsys_0\in\Gamma_{\hat\alpha}$ 
and the DMPC scheme stabilizes the
origin exponentially.
\end{theorem}

\begin{proof}
  The result \eqref{eq:err_to_stop_crit} of
  Lemma~\ref{lem:err_to_stop_crit} can be expressed in terms of the
  optimal cost using the quadratic bound \eqref{eq:app_J_lower_bound}, i.e.\
  \begin{equation}
    \|\Delta\xmpc_F^{q_k}(\dt;\xsys_k)\|_{L^\infty} \le 
    \frac{D d}{\sqrt{m_J}}\sqrt{J^*(\xsys_k)} \,.
  \end{equation}
  Combining this relation with \eqref{eq:cost_subopt_sol} of
  Lemma~\ref{lem:cost_subopt_sol} results in   
  \begin{equation} \label{eq:factor_cost_decrease}
    J^*(\xmpc_F^{q_k}(\dt;\xsys_k)) \le C_J J^*(\xsys_k)  \quad
    \text{with} \quad
    C_J = (1-a) + \frac{ b \, d \, D }{ \sqrt{m_J} } + \frac{ c \, d^2 \, D^2 }{ m_J }\,.
  \end{equation}
  In particular, $ C_J < 1 $ if 
  \begin{equation} \label{eq:upperbound_d}
    d < \frac{ \sqrt{m_J} }{ 2 \, c \, D} \left( \sqrt{b^2 + 4 ac}-b \right)\,,
  \end{equation}
  which implies exponential decay of both the optimal cost
  $J^*(\xsys_k) \le C_J^k J^*(\xsys_0) \le C_J^k \hat\alpha$ and the
  error $\|\Delta\xmpc_F^{q_k}(\dt;\xsys_k)\|\le D d\|\xsys_k\|$
  with 
  \begin{equation}
    \label{eq:upperbound_x_lin_conv}
    \|\xsys_k\|\le \frac{1}{\sqrt{m_J}}C_J^{k/2}\sqrt{\hat\alpha}\,.
  \end{equation}
  Exponential stability in continuous time requires to look
  at the state trajectory in closed-loop, i.e.\
  $\xsys(t) = \xsys(t_k+\tau) = x_F^{q_k}(\tau;\xsys_k)$,
  $k \in \Nbb_0, \, \tau \in [0, \Delta t] $.
  Employing the triangle inequality, the bound
  \eqref{eq:app:upper_bound_xopt}, the definition of the stopping
  criterion \eqref{eq:ADMM_convcrit} and the bounds on
  $d$ and $\norm{x_k}$, 
  \begin{align} \nonumber
    \|x_F^{q_k}(\tau;\xsys_k)\| 
    &\le \|\xmpc^*(\tau;\xsys_k)\| + \|\Delta x_F^{q_k}(\tau;\xsys_k)\|
    \le \left( e^{\hat L\tau} + d D \right) \|\xsys_k\|
    \\ &\le \label{eq:bound_actual_traj}
    \left( e^{\hat L\dt} + d D \right)
    \frac{1}{\sqrt{m_J}}C_J^{k/2}\sqrt{J^*(\xsys_0)}
  \end{align}
  bounds the trajectory for $ \tau \in [0,\Delta t] $ in every MPC
  step $ k $. The asymptotic decay of this bound with increasing $ k $
  implies the existence of an exponential envelope function with
  constants $ \gamma_1, \gamma_2 > 0 $ such that 
  $ \norm{ x(t) } \leq \gamma_1 e^{-\gamma_2 t} $. 
\end{proof}

The proof of Theorem~\ref{th:exp_stab_DMPC} shows that 
an explicit value for $d$ can be computed,
cf.~\eqref{eq:upperbound_d}. However, this value is usually too
conservative to be used for design purposes due to the different
Lipschitz and continuity estimates that are involved.
Nevertheless, Theorem~\ref{th:exp_stab_DMPC} states that the constant
$d>0$ in the ADMM stopping criterion~\eqref{eq:ADMM_convcrit} can
always be chosen small enough to ensure stability and incremental
improvement of the DMPC scheme. The explicit choice of $d$ depends on
the DMPC problem at hand and also represents a trade-off between
control performance and the number of ADMM iterations in each MPC
step. An interesting statement on the number of iterations
$q_k$ that are actually needed to fulfill the stopping criterion
\eqref{eq:ADMM_convcrit} follows as outcome of the next section.

\subsection{Relaxed ADMM convergence assumptions}
\label{sec:relaxed_assumptions}

The stability results in Section~\ref{sec:stability_subsection} assume
linear convergence of the ADMM algorithm
(Assumption~\ref{as:linear_convergence}).
Though linear convergence for ADMM has recently been shown in
different settings, see the literature overview at the end of
Section~\ref{sec:DMPC_preliminaries}, it 
can be restrictive in practice. 
To this end, we consider a relaxation of
Assumption~\ref{as:linear_convergence} in form of 
R-linear convergence.

\begin{assumption}
  \label{as:linear_conv_Rlinear}
  The ADMM iterates are R-linearly convergent in the sense of
  \begin{equation}
    \label{eq:linear_conv_Rlinear}
    \begin{Vmatrix}
      z^q(\cdot;\xsys_k)-z^*(\cdot;\xsys_k)
      \\
      \mu^q(\cdot;\xsys_k)-\mu^*(\cdot;\xsys_k)
    \end{Vmatrix}_{L^\infty}
    \le C_0 C^q 
    \begin{Vmatrix}
      z^0(\cdot;\xsys_k)-z^*(\cdot;\xsys_k)
      \\
      \mu^0(\cdot;\xsys_k)-\mu^*(\cdot;\xsys_k)
    \end{Vmatrix}_{L^\infty}
  \end{equation}
  for some $C_0>0$ and $C\in[0,1)$.
\end{assumption}

Assumption~\ref{as:linear_conv_Rlinear} is clearly weaker than
Assumption~\ref{as:linear_convergence}, as the right-hand side of
\eqref{eq:linear_conv_Rlinear} describes an envelope function that
allows for temporary increases of the iterates on the left-hand side of
\eqref{eq:linear_conv_Rlinear}.
R-linear convergence for ADMM has been investigated in the literature,
for instance, in \cite{Nishihara2015,Ghadimi:TAC:2015,Deng2016,Hong:MP:2017}, 
also see \cite{Han:2015} for a corresponding literature overview.
An example for R-linear ADMM convergence will also be part of the
simulation results in Section~\ref{sec:simulation_results}.

\begin{theorem} \label{th:exp_stab_DMPC_Rlinear_conv}
Suppose that Assumptions
\ref{as:clf_on_omega}-\ref{as:ADMM_boundedness} and
\ref{as:linear_conv_Rlinear} hold. Then, there exists a fixed number
of iterations $\hat q>0$ for the ADMM algorithm without evaluating the
stopping criterion~\eqref{eq:ADMM_convcrit}
such that the optimal cost \eqref{eq:Jopt} and the error
\eqref{eq:actualerror} decay exponentially for all 
$\xsys_0\in\Gamma_{\hat\alpha}$ 
and the DMPC scheme stabilizes the
origin exponentially.
\end{theorem}

\begin{proof}
To simplify the following considerations, we use the compact notation 
$z^{q|k}:=z^q(\cdot;\xsys_k)$ as well as $\delta z^{q|k}:=z^{q|k}-z^{q-1|k}$ and 
$\Delta z^{q|k}:=z^{q|k}-z^{*|k}$ for $z$ as well as for the
multipliers $\mu$. 
The R-linear convergence property \eqref{eq:linear_conv_Rlinear}
together with Minkowski's inequality yields the bound 
\begin{align} \label{eq:proof_Rlinear_1}
\begin{Vmatrix} \delta z^{q|k} \\ \delta\mu^{q|k} \end{Vmatrix}_{L^\infty}
\!\!\!\le 
\begin{Vmatrix} \Delta z^{q|k} \\ \Delta\mu^{q|k} \end{Vmatrix}_{L^\infty}
\!\!\! +
\begin{Vmatrix} \Delta z^{q-1|k} \\ \Delta\mu^{q-1|k} \end{Vmatrix}_{L^\infty}
\!\!\!\le C_0 C^{q-1}(1+C) \begin{Vmatrix} \Delta z^{0|k} \\ \Delta\mu^{0|k} \end{Vmatrix}_{L^\infty}.
\end{align}
Hence, given the current system state 
$\xsys_k=[\xsys_{k,i}]_{i\in\Vcal}$,
the previous stopping criterion~\eqref{eq:ADMM_convcrit} can
always be satisfied for $q\ge q_k$ with
\begin{equation} 
\label{eq:proof_Rlinear_2}
q_k := 1 + \left\lceil 
  \mylog_C \left( \frac{d \mymin_{i\in\Vcal}\|\xsys_{k,i}\|}{C_0(1+C)} 
    \begin{Vmatrix} \Delta z^{0|k} \\ \Delta\mu^{0|k} \end{Vmatrix}_{L^\infty}^{-1}
\right)
\right\rceil \,.
\end{equation}
As a consequence of \eqref{eq:proof_Rlinear_1}, this
also implies
\begin{equation}
\label{eq:proof_Rlinear_3}
\begin{Vmatrix} \Delta z^{q-1|k} \\ \Delta\mu^{q-1|k} \end{Vmatrix}_{L^\infty}
\le d \|\xsys_k\| \le \frac{d}{1-C} \|\xsys_k\| 
\end{equation}
with $C\in[0,1)$. The latter bound can be used to substitute the
linear convergence bound in the third line
of~\eqref{eq:bound_err_act_3}. Hence, if $q\ge q_k$ for all MPC
steps $k$ and provided that $q_k$ is finite, the results of
Lemma~\ref{lem:err_to_stop_crit} and Theorem~\ref{th:exp_stab_DMPC}
also hold for the R-linear convergence case.

The boundedness of a fixed upper iteration limit 
$\hat q := \mymax_{k\in\Nbb_0^+} q_k$
for all $k$
can be shown by induction. In the initial MPC step $k=0$ and according to
\eqref{eq:proof_Rlinear_2}, $q_0$ depends on the initial state  
$\xsys_0\in\Gamma_{\hat\alpha}$ and the errors 
$(\Delta z^{0|0},\Delta\mu^{0|0})$ of the initial trajectories
$\mu^{0|0}\in\Mcal_\alpha^{\sigma_\mu}$ and 
$z^{0|0}\in\Zcal_\alpha^{\sigma_z}$. The compactness of the sets 
$\Gamma_{\hat\alpha}$ and $\Mcal_\alpha^{\sigma_\mu}$,
$\Zcal_\alpha^{\sigma_z}$
guarantees that there exists an upper bound on $q_0$ 
that holds for all admissible $\xsys_0$ and $(\Delta
z^{0|0},\Delta\mu^{0|0})$. 
Following the lines of the proof of Theorem~\ref{th:exp_stab_DMPC}
then shows that $J^*(\xsys_1)\le C_J J^*(\xsys_0)$ with $C_J<1$
provided that $d$ is sufficiently small.

For an arbitrary MPC step $k$, assume that $q_k$ iterations yield
\begin{equation}
\label{eq:proof_Rlinear_4}
\begin{Vmatrix} \delta z^{q_k|k} \\ \delta\mu^{q_k|k} \end{Vmatrix}_{L^\infty}
\le d \|\xsys_k\| \,, \quad
\begin{Vmatrix} \Delta z^{q_k-1|k} \\ \Delta\mu^{q_k-1|k} \end{Vmatrix}_{L^\infty}
\le d \|\xsys_k\|
\end{equation}
and MPC step $k+1$ is initialized according to
\eqref{eq:multipl_initialization}. As reasoned above, the proof of
Theorem~\ref{th:exp_stab_DMPC} then yields 
$J^*(\xsys_{k+1})\le C_J J^*(\xsys_k)$.
Next, consider the following bound in step $k+1$
\begin{align} \nonumber
\begin{Vmatrix} \delta z^{q|k+1} \\ \delta\mu^{q|k+1} \end{Vmatrix}_{L^\infty}
&\le
\begin{Vmatrix} \Delta z^{q|k+1} \\ \Delta\mu^{q|k+1} \end{Vmatrix}_{L^\infty}
+
\begin{Vmatrix} \Delta z^{q-1|k+1} \\ \Delta\mu^{q-1|k+1} \end{Vmatrix}_{L^\infty}
\\ \nonumber
&\le C_0 C^{q-1}(1+C) 
\begin{Vmatrix} \Delta z^{0|k+1} \\ \Delta\mu^{0|k+1} \end{Vmatrix}_{L^\infty}
\\ \label{eq:proof_Rlinear_5} &\le
C_0 C^{q-1}(1+C) \left( 
\begin{Vmatrix} \Delta z^{q_k|k} \\ \Delta\mu^{q_k|k} \end{Vmatrix}_{L^\infty}
+
\begin{Vmatrix} z^{*|k+1}-z^{*|k} \\ \mu^{*|k+1}-\mu^{*|k} \end{Vmatrix}_{L^\infty}
\right)
\end{align}
that follows from Minkowski's inequality, the R-linear convergence
property in Assumption~\ref{as:linear_conv_Rlinear}, and the warm start 
\eqref{eq:multipl_initialization} for MPC step $k+1$. 
Regarding the second bracket term in the last line of
\eqref{eq:proof_Rlinear_5}, Assumption~\ref{as:unique_bounded_multipl}
implies that there exist finite Lipschitz constants $L_z,L_\mu>0$ over
the compact set $\Gamma_\alpha$ such that 
\begin{equation}
\label{eq:proof_Rlinear_6}
\begin{Vmatrix}
z^{*|k+1}-z^{*|k} \\ \mu^{*|k+1}-\mu^{*|k}
\end{Vmatrix}_{L^\infty}
\le 
(L_z+L_\mu) \|\xsys_{k+1}-\xsys_k\|
\le (L_z+L_\mu) (\|\xsys_k\|+\|\xsys_{k+1}\|)
\end{equation}
for all $\xsys_k\in\Gamma_{\hat\alpha}$ (remember that 
$\xsys_{k+1}\in\Gamma_\alpha$ as shown in Lemma~\ref{lem:cost_subopt_sol}). 
The norm of the current state $\|\xsys_k\|$ can be related to 
$\|\xsys_{k+1}\|=\|\xmpc_F^{q_k}(\dt;\xsys_k)\|$
with
\begin{align} \nonumber
\|\xmpc_F^{q_k}(\dt;\xsys_k)\| &= 
\|\xmpc^*(\dt;\xsys_k) + \Delta\xmpc_F^{q_k}(\dt;\xsys_k)\|
\\ \nonumber &\ge
\|\xmpc^*(\dt;\xsys_k)\| - \|\Delta\xmpc_F^{q_k}(\dt;\xsys_k)\|
\\ &\ge \label{eq:proof_Rlinear_7}
\left( e^{-\hat L \dt}-Dd\right) \|\xsys_k\|\,,
\end{align}
where the reverse triangle inequality together with~\eqref{eq:app:lower_bound_xopt}
and \eqref{eq:err_to_stop_crit} was used (note that
Lemma~\ref{lem:err_to_stop_crit} holds as reasoned above). 
If $d$ is sufficiently small such that $d<e^{-\hat L \dt}/D$, then
$\|\xsys_k\|$ is upper bounded by $\|\xsys_{k+1}\|$ according to
\begin{equation}
\label{eq:proof_Rlinear_8}
\|\xsys_{k}\| \le \frac{1}{e^{-\hat L \dt}-Dd} \|\xsys_{k+1}\| \,.
\end{equation}
This shows together with \eqref{eq:proof_Rlinear_4} that
\eqref{eq:proof_Rlinear_5} can be expressed in 
terms of $\|\xsys_{k+1}\|$ and the iteration number $q$
\begin{equation}
\label{eq:proof_Rlinear_9}
\begin{Vmatrix} \delta z^{q|k+1} \\ \delta\mu^{q|k+1} \end{Vmatrix}_{L^\infty}
\le
C^{q-1} E_3 \|\xsys_{k+1}\|
\,, \quad
\begin{Vmatrix} \Delta z^{q-1|k+1} \\ \Delta\mu^{q-1|k+1} \end{Vmatrix}_{L^\infty}
\le
C^{q-1} E_3 \|\xsys_{k+1}\|
\end{equation}
with the constant 
$E_3:=C_0(1+C)\left(\frac{d+L_\mu+L_z}{e^{-\hat L \dt}-Dd}+L_\mu+L_z\right)$ 
and the second bound directly following from 
\eqref{eq:proof_Rlinear_5} as reasoned before.
Hence, if $q$ satisfies $q\ge q_{k+1}$ with
\begin{equation}
\label{eq:proof_Rlinear_10}
q_{k+1} := 1 + \left\lceil 
  \mylog_C \left( \frac{d}{E_3} \right) \right\rceil \,,
\end{equation}
the bounds \eqref{eq:proof_Rlinear_9} become
\begin{equation}
\begin{Vmatrix} \delta z^{q|k+1} \\ \delta\mu^{q|k+1} \end{Vmatrix}_{L^\infty}
\le d\|\xsys_{k+1}\| \,, \quad
\begin{Vmatrix} \Delta z^{q-1|k+1} \\ \Delta\mu^{q-1|k+1} \end{Vmatrix}_{L^\infty}
\le d \|\xsys_{k+1}\|\,,
\end{equation}
which completes the induction step. Note that $q_{k+1}$ in
\eqref{eq:proof_Rlinear_10} is independent of $k$. This shows that
there exists a fixed number of iterations, given by 
$\hat q=\mymax\{q_0,q_{k+1}\}$ such that exponential stability 
and exponential decay of the error \eqref{eq:actualerror} follows
along the lines of Theorem~\ref{th:exp_stab_DMPC}.
\end{proof}

Theorem~\ref{th:exp_stab_DMPC_Rlinear_conv} states that using a sufficiently
large number of iterations $\hat q$ instead of the stopping criterion 
\eqref{eq:ADMM_convcrit} is sufficient in order to
guarantee exponential stability of the DMPC scheme. In fact, 
showing stability based on the stopping criterion requires the linear
convergence assumption (see Theorem~\ref{th:exp_stab_DMPC}). 
On the one hand, using a fixed number of ADMM iterations avoids the
global communication step that is required to check the satisfaction
of \eqref{eq:ADMM_convcrit} for all MPC agents $i\in\Vcal$. 
On the other hand, the number of iterations $q_k$
for satisfying the stopping criterion~\eqref{eq:ADMM_convcrit} will in
practice be lower for achieving the same performance. 

The following corollary is a straightforward consequence of
Theorem~\ref{th:exp_stab_DMPC_Rlinear_conv} applied to the linear
convergence case of Section~\ref{sec:ADMM_distr}.

\begin{corollary}
\label{cor:bounded_iteration_lin_conv}
Suppose that the assumptions of Theorem~\ref{th:exp_stab_DMPC} in the
linear ADMM convergence case are satisfied. 
Then, there exists an upper bound $\hat q=\hat q(d)<\infty$ on the iteration number
$q_k$ for satisfying the stopping criterion~\eqref{eq:ADMM_convcrit}, 
i.e.\ $q_k\le \hat q$ for all $k\in\Nbb_0^+$.
\end{corollary}

This result in connection with the stopping
criterion~\eqref{eq:ADMM_convcrit} coincides with numerical
experiences showing that $q_k$ is typically largest for the
first MPC step $k=0$ and subsequently converges to a constant limit
for increasing $k$ that depends on the choice of $d>0$, i.e. on the
strictness of the stopping criterion. The numerical examples in the
following Section~\ref{sec:simulation_results} will further highlight
this behavior.

\section{Simulation Results}\label{sec:simulation_results}

Two simulation examples are used to demonstrate the performance of the
ADMM-based DMPC scheme. The general solution behavior and influence of
the stopping criterion are investigated for a system of coupled Van
der Pol oscillators, whereas the scalability for larger systems is
shown for a spring mass system with a variable number of masses. 
The problems are implemented within an object-oriented framework in
\textsc{Matlab}. The local OCPs \eqref{eq:cost:MPClocal} are solved
with the toolbox \textsc{GRAMPC}~\cite{Kaepernick2014}.

\subsection{Van der Pol oscillators}\label{sec:VdP_oscillators}

The following system describes a set of three coupled Van der Pol
oscillators \cite{Dutra2003,Dunbar2007} 
\begin{subequations}
  \label{eq:example_oscillators}
  \begin{alignat}{2}
    \ddot{\phi}_1 &= 0.1 (1-5.25 \phi_1^2) \dot{\phi}_1 - \phi_1 + u_1
    \\
    \ddot\phi_2 &= 0.01 (1-6070 \phi_2^2) \dot{\phi}_2 - 4 \phi_2
    + 0.057 \phi_1\dot{\phi}_1 + 0.1 (\dot{\phi}_2 - \dot{\phi}_3) + u_2
  \\
  \ddot\phi_3 &=
    0.01 (1-192 \phi_3^2) \dot{\phi}_3 - 4 \phi_3 + 0.057
    \phi_1\dot{\phi}_1 + 0.1 (\dot{\phi}_3 - \dot{\phi}_2) + u_3 
\end{alignat}
\end{subequations}
with the controls $u_i$
subject to the constraints
$\left| u_i\right|  \leq 1 \; \text{rad/}\text{s}^2$,  
$i \in\Vcal = \{1,2,3\}$. 
The single subsystems can be written in the form \eqref{eq:dyn_subsys}
with the states $x_i = [\phi_i,\dot{\phi}_i]^\TT$ and 
the neighborhood of each subsystem given by
\begin{alignat}{3}\label{eq:OsciNeighborSets}
\nonumber
&\Nto 1 = \{\} \quad &&\Nto 2 = \{1,3\} \quad &&\Nto 3 = \{1,2\} \\
&\Ngets 1 = \{2,3\} \quad &&\Ngets 2 = \{3\} \quad &&\Ngets 3 = \{2\}.
\end{alignat}
Note that the original representation of the dynamics \cite{Dunbar2007} 
was transformed into \eqref{eq:example_oscillators}, in order to shift
the setpoint to be stabilized to the origin. 

The cost functions for the MPC
formulation~\eqref{eq:gl_ocp_in_gl_quant} are chosen in quadratic form 
\begin{equation}\label{eq:ex:quadraticCost}
  V_i(x_i) = x_i^\TT P_i x_i \,, \quad
  l_i(x_i,u_i) = \gamma (x_i^\TT Q x_i + r u_i^2)\,, \quad i \in \Vcal\,.
\end{equation}
The weights  
$Q = \diag (30,30)$ and $r=0.1$ as well as the weighting matrices $P_i
= P_i^\TT > 0$ for the terminal costs are taken from \cite{Dunbar2007}. 
In particular, $P_i$ is obtained by solving the Lyapunov equation
\begin{equation}\label{eq:ex:LyapForPi}
	P_i A_{\text{d}i} + A_{\text{d}i}^\TT P_i = \hat{Q}_i
\end{equation}
with $ \hat{Q}_i = Q + r K_i^\TT K_i $ and 
$(A_{\text{d}i}, B_{\text{d}i})$ given by the jacobians 
$(\frac{\partial f_i}{\partial x_i}, \frac{\partial f_i}{\partial u_i})$
evaluated at the origin.
The block-diagonal matrix $ K = \diag(K_1,K_2,K_3)$ determines the linear
feedback law $ u = \kappa_V(x) = K x $ with $ K_1 = [3.6, 5.3] $ and $
K_2 = K_3 = [2.0, 5.0]$, cf.\  \cite{Dunbar2007}. 
The additional weighting parameter $ \gamma= 0.2 $ in
\eqref{eq:ex:quadraticCost} ensures Assumption \ref{as:clf_on_omega},
since compared to \cite{Dunbar2007} not only 
$\dot{V}(x,\kappa_V(x)) \leq 0$ but
$\dot{V}(x,\kappa_V(x)) + l((x,\kappa_V(x))) \leq 0 $ must be
satisfied.
Finally, the MPC horizon and the sampling time are set to $T = 3$ s and
$\Delta t = 0.1$ s, respecticely. 

\begin{figure}
	\centering
	\setlength\fwidth{14.8cm}
	\setlength\fheight{0.75\fwidth}
	\input{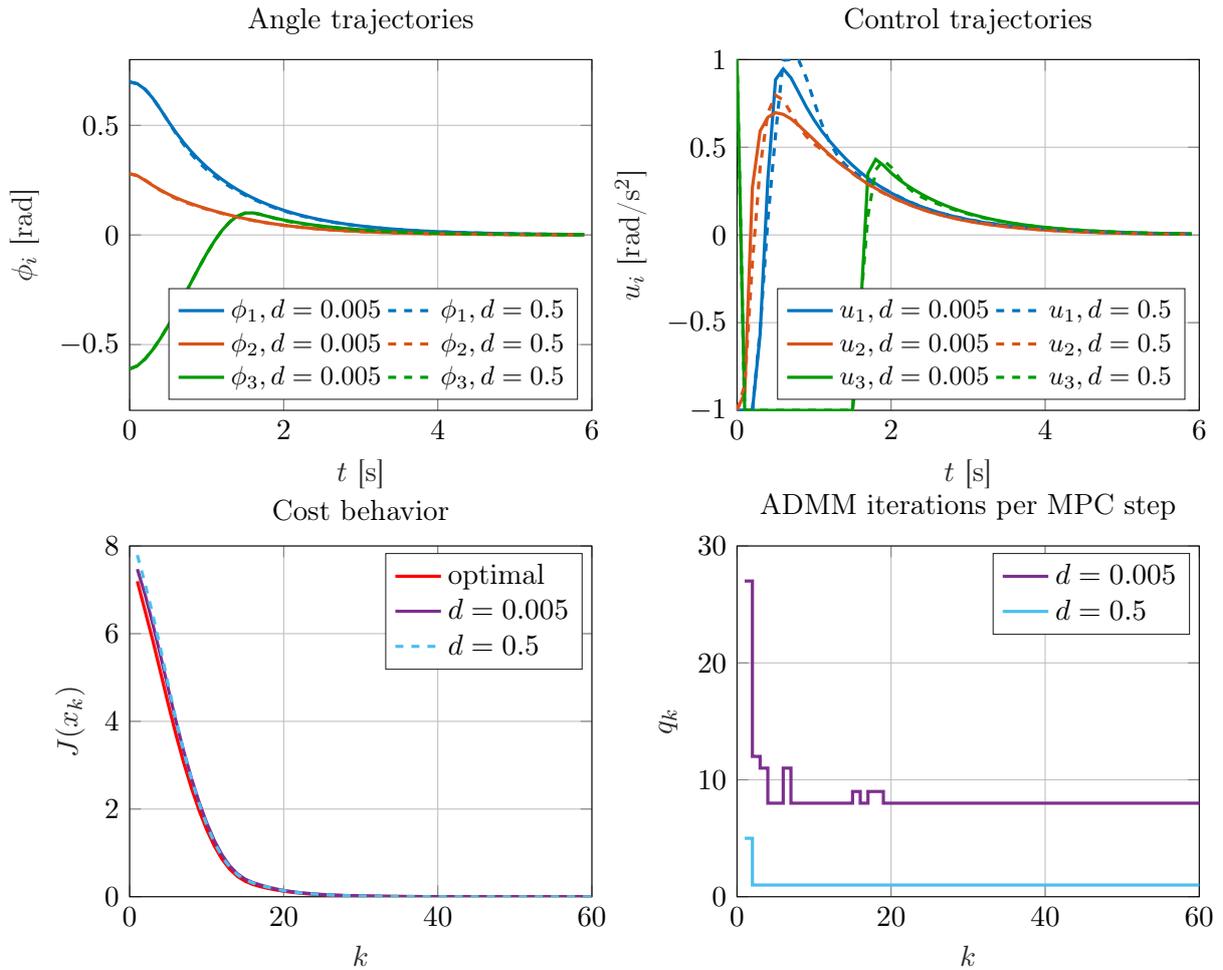}
	\caption{Solution behavior for different stopping criteria.}
	\label{fig:varD}
\end{figure}

The initial values for the DMPC simulation are chosen as 
$\phi_{1,0} = 0.698,\phi_{2,0}=0.279 ,
\phi_{3,0}=-0.611 $ and $\dot{\phi}_{i,0}=0$ corresponding to
\cite{Dunbar2007} in the untransformed system representation. 
The upper half of Figure~\ref{fig:varD} shows the angles
$\phi_i$ and constrained controls $u_i$
for two values of the constant $d$ in the stopping
criterion~\eqref{eq:ADMM_convcrit}. 
In addition, the lower half of Figure~\ref{fig:varD}
shows the exponential decay of the optimal cost value and the
iterations $q_k$ in each MPC step for satisfying the stopping criterion.
The difference between the two values $d=0.5$ and $d=0.005$ is
particularly visible in terms of the controls $u_i$ and ADMM
iterations $q_k$. The ``softer'' setting $d=0.5$ leads to 
slightly delayed trajectories with larger magnitudes, while the 
stricter choice $d=0.005$ requires a clearly higher number of
iterations $q_k$ than in case of the more relaxed setting $d=0.5$.

In both cases, it is visible that $q_k$ in the right lower part of
Figure~\ref{fig:varD} is bounded and converges to a fixed limit,
which is inline with Corollary~\ref{cor:bounded_iteration_lin_conv}.
Figure~\ref{fig:avgqkOverd} additionally shows the average number of
iterations $q_k$ for different values $d$ which again illustrates
the higher computational effort for satisfying the stricter stopping
criterion~\eqref{eq:ADMM_convcrit} as $d$ is decreased.

An interesting question concerns the convergence behavior of the ADMM
iterations in view of the linear and R-linear convergence assumptions
\eqref{eq:linear_convergence} and \eqref{eq:linear_conv_Rlinear},
respectively. Figure~\ref{fig:residualRlin} shows the residual norm
\begin{equation}
  \label{eq:ADMM_residual}
  \Delta^q_{res} =
  \begin{Vmatrix}
    z^q(\cdot;\xsys_k)-z^*(\cdot;\xsys_k)
    \\
    \mu^q(\cdot;\xsys_k)-\mu^*(\cdot;\xsys_k)
  \end{Vmatrix}_{L^\infty}
\end{equation}
plotted over the iterations $q$ for different initializations
of the ADMM algorithm. 
Although linear convergence is not guaranteed in each iteration $q$,
the algorithm is at least R-linearly convergent. The corresponding
envelope function according to the right-hand side
of~\eqref{eq:linear_conv_Rlinear} in
Assumption~\ref{as:linear_conv_Rlinear} is shown in
Figure~\ref{fig:residualRlin}.

\begin{figure}
	\centering
	\input{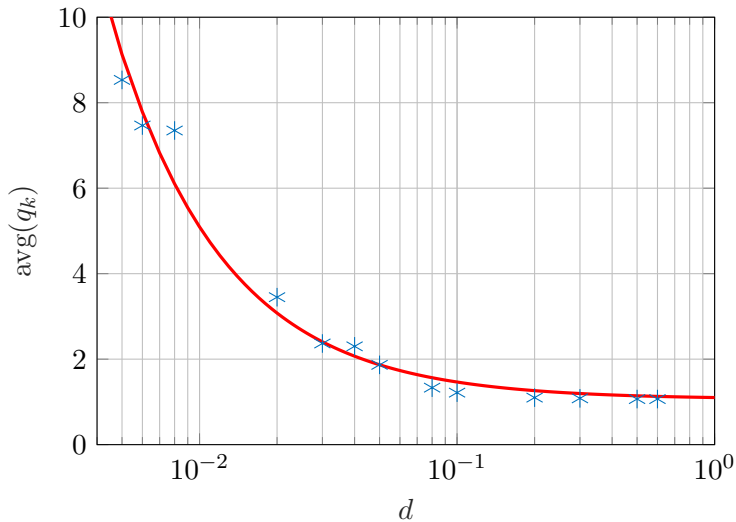}
	\caption{Average ADMM iteration numbers $q_k$ for different values
          of $d$ in the stopping criterion~\eqref{eq:ADMM_convcrit}.} 
	\label{fig:avgqkOverd}
\end{figure}
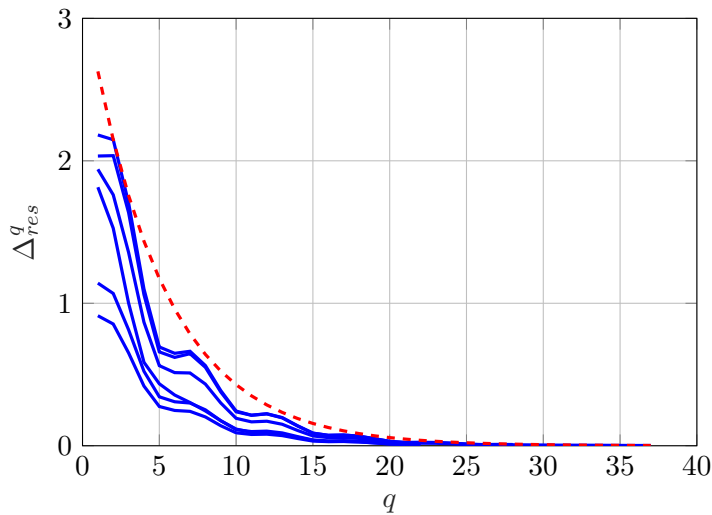
\begin{figure}
	\centering
%
%
\begin{tikzpicture}

\begin{axis}[%
width=0.951\fwidth,
height=\fheight,
at={(0\fwidth,0\fheight)},
scale only axis,
xmin=0,
xmax=40,
xlabel style={font=\color{white!15!black}},
xlabel={$q$},
ymin=0,
ymax=3,
ylabel style={font=\color{white!15!black}},
ylabel={$\Delta^q_{res}$},
axis background/.style={fill=white},
xmajorgrids,
ymajorgrids
]
\addplot [color=blue, line width=1.2pt, forget plot]
  table[row sep=crcr]{%
1	0.912698863641407\\
2	0.854464189868871\\
3	0.651459026231364\\
4	0.419027812495113\\
5	0.275392912925149\\
6	0.247730552809886\\
7	0.241350555954783\\
8	0.202172413086743\\
9	0.140562012928065\\
10	0.0908198036584573\\
11	0.0792003820862386\\
12	0.0811571589443086\\
13	0.0711771443139968\\
14	0.0509251569528822\\
15	0.0328081165350977\\
16	0.028752183176469\\
17	0.0298393492428475\\
18	0.0267693047141792\\
19	0.0195539104948255\\
20	0.0123537619188452\\
21	0.0102276542342191\\
22	0.0106727646367774\\
};
\addplot [color=blue, line width=1.2pt, forget plot]
  table[row sep=crcr]{%
1	1.14154746790124\\
2	1.06891515070866\\
3	0.813943493416046\\
4	0.522885424838474\\
5	0.34345922816712\\
6	0.308662846952329\\
7	0.299871491330336\\
8	0.252122292035593\\
9	0.175795285605851\\
10	0.113286873709305\\
11	0.0993417869315552\\
12	0.102456687686177\\
13	0.0900957721967874\\
14	0.0641816427934822\\
15	0.0405697547018966\\
16	0.0340015183289474\\
17	0.035953867155291\\
18	0.0329869128421745\\
19	0.0245950580221763\\
20	0.0157708870671602\\
21	0.0121282212614627\\
22	0.0128034467067833\\
};
\addplot [color=blue, line width=1.2pt, forget plot]
  table[row sep=crcr]{%
1	2.03351592994246\\
2	2.03604975447899\\
3	1.63663802266013\\
4	1.06106089279129\\
5	0.658656294077157\\
6	0.62003578623949\\
7	0.645956400789584\\
8	0.550412301728539\\
9	0.377656317049481\\
10	0.238161973224923\\
11	0.213882463863707\\
12	0.223630593625109\\
13	0.195777158254257\\
14	0.138400357065956\\
15	0.0880697299628621\\
16	0.0735274911496271\\
17	0.0753990258062526\\
18	0.0671704874380689\\
19	0.0502533297387073\\
20	0.0307154006225384\\
21	0.0225276639945934\\
22	0.0237094217341066\\
23	0.023127271359691\\
24	0.018396037833854\\
25	0.0119392186860554\\
26	0.00842125963564332\\
27	0.00943210476929678\\
28	0.00860672431304019\\
};
\addplot [color=blue, line width=1.2pt, forget plot]
  table[row sep=crcr]{%
1	1.94038413212646\\
2	1.7613234246108\\
3	1.35685789662288\\
4	0.86711737797871\\
5	0.562276930119723\\
6	0.513240359250711\\
7	0.510819764022933\\
8	0.432855687472075\\
9	0.300272869240767\\
10	0.192652352094394\\
11	0.167637087045983\\
12	0.172107158052459\\
13	0.150774247069203\\
14	0.107406301432751\\
15	0.0687459779717942\\
16	0.0570608275600781\\
17	0.0580804822194096\\
18	0.0514048530264534\\
19	0.0375693334484354\\
20	0.0234840615680821\\
21	0.0171414385747934\\
22	0.01766122205169\\
23	0.0172526973364207\\
24	0.0136278889423867\\
25	0.00890749062090952\\
26	0.00629880019954922\\
27	0.00691158926810387\\
28	0.00623015192868737\\
};
\addplot [color=blue, line width=1.2pt, forget plot]
  table[row sep=crcr]{%
1	1.81457904551701\\
2	1.52672512343931\\
3	1.00144349343558\\
4	0.58587096299915\\
5	0.434902450652059\\
6	0.356068858336886\\
7	0.303299515982649\\
8	0.244189722535825\\
9	0.175731680304659\\
10	0.117182922270444\\
11	0.0934804007958409\\
12	0.0916080857907798\\
13	0.0811106009000181\\
14	0.0589675198989985\\
15	0.0377429082574434\\
16	0.0301597418894592\\
17	0.0304937988076567\\
18	0.0274262592876438\\
19	0.0203403289355565\\
20	0.0131123806996886\\
21	0.0097935107199212\\
};
\addplot [color=blue, line width=1.2pt, forget plot]
  table[row sep=crcr]{%
1	2.18139076742078\\
2	2.14812205992571\\
3	1.70435741324106\\
4	1.10065930152443\\
5	0.693370121498416\\
6	0.648706337029826\\
7	0.663071300566408\\
8	0.563046684392026\\
9	0.387535796821662\\
10	0.243514962447353\\
11	0.214402600744187\\
12	0.224602997981582\\
13	0.197759468672074\\
14	0.140709732318977\\
15	0.0887712824661045\\
16	0.0727270780887904\\
17	0.0749951523089556\\
18	0.0673087400178103\\
19	0.0487398481290787\\
20	0.0310278135884805\\
21	0.0229873406946709\\
22	0.0231147954666693\\
23	0.021982533673211\\
24	0.017160498878558\\
25	0.0108434322282097\\
26	0.0071338873297841\\
27	0.00699479991479145\\
28	0.00698967474886727\\
29	0.0061503279158777\\
30	0.00511726557039253\\
31	0.00428557952831962\\
32	0.00395797485795663\\
33	0.0036685642615471\\
34	0.00311155601250641\\
35	0.00241471685397133\\
36	0.00192420318843281\\
37	0.00185863178584714\\
};
\addplot [color=red, dashed, line width=1.2pt, forget plot]
  table[row sep=crcr]{%
1	2.6276683769871\\
2	2.14812764217308\\
3	1.75610149571425\\
4	1.43561881645455\\
5	1.17362315970245\\
6	0.95944083847523\\
7	0.7843460781461\\
8	0.64120552892126\\
9	0.524187653606919\\
10	0.428525151141779\\
11	0.350320737044274\\
12	0.286388368281888\\
13	0.234123444073471\\
14	0.191396694613209\\
15	0.156467434749363\\
16	0.127912648577979\\
17	0.104569015862262\\
18	0.0854855184374975\\
19	0.0698846957893676\\
20	0.0571309713602929\\
21	0.0467047591994696\\
22	0.0381812960631107\\
23	0.0312133365859527\\
24	0.0255170065263781\\
25	0.0208602377472281\\
26	0.0170533137741311\\
27	0.0139411407579769\\
28	0.0113969289610182\\
29	0.00931702735073345\\
30	0.0076167008631209\\
31	0.00622667830138974\\
32	0.00509033023165247\\
33	0.00416136190326259\\
34	0.00340192720351341\\
35	0.00278108680932823\\
36	0.00227354772113628\\
37	0.00185863282762201\\
};
\end{axis}
\end{tikzpicture}%
	\caption{Progress of the ADMM
          residual~\eqref{eq:ADMM_residual}
          for different initializations and envelope
          function corresponding to the R-linear convergence case \eqref{eq:linear_conv_Rlinear}.}
	\label{fig:residualRlin}
\end{figure}

\subsection{Scalable spring mass system}

The second example is a scalable spring-mass system that is used to
investigate the scalalility and behavior of the ADMM iterations $q_k$
for increasing numbers of subsystems $N$.
The setup is similar to \cite{Riverso2015}, where a damped spring-mass
system was considered.
Figure~\ref{fig:springMass} shows the system for $N=3$.
The masses $m_i$ are connected by springs with spring constant
$c$. The displacements $s_i$ denote the deviation from 
the equilibrium point. 
\begin{figure}
	\centering
	{
		\normalsize
		\resizebox{0.49\textwidth}{!}{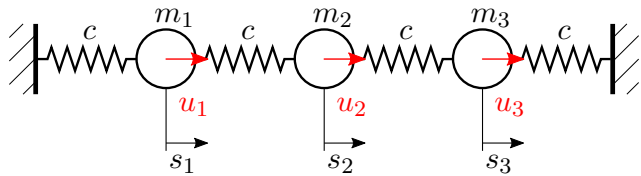}}
	\caption{Coupled spring-mass-system for $N=3$ agents.}
	\label{fig:springMass}
\end{figure}

The equations of motion for each mass read 
\begin{subequations}
\label{eq:springMass}
\begin{align}
m_i \ddot s_i &= c ( s_{j} - 2s_{i} ) + u_{i}
\end{align}
if coupled with only one neighbor $j$, cf. $m_1$ and $m_3$ 
in Figure~\ref{fig:springMass}, or 
\begin{align}
m_i \ddot s_i &= c ( s_{j} + s_{k} - 2s_{i} ) + u_{i}
\end{align}
\end{subequations}
if coupled with two neighbors $j$ and $k$, cf. $m_2$ in
Figure~\ref{fig:springMass}. We assume normalized quantities in the
dynamics~\eqref{eq:springMass} for the sake of simplicity. 
Each mass $i$ can be
effected by a control $u_i$ subject to the constraint 
$\left| u_i\right|  \leq 1$, $i \in \Vcal = \{ 1, \dots,N \}$. 
By defining the state vector
$x_i = [ s_i, \, \dot{s}_i ]^\TT$, the dynamics can be written as
\eqref{eq:dyn_subsys} with the sending and receiving neighbor sets 
$ \Nto i = \Ngets i$, $i \in \Vcal$.

The explicit values for the masses $m_i$ and the initial (stationary)
values of $ s_i(0)=s_{i,0}$ are randomly distributed over
the intervals $[5,10]$ and $[-0.5,0.5]$,
respectively, 
whereas the spring constant is  set to $ c = 0.5 $.
The MPC formulation follows the first example in Section~\ref{sec:VdP_oscillators} with \eqref{eq:ex:quadraticCost} and \eqref{eq:ex:LyapForPi}. For $\gamma = 0.5$ and any stabilizing feedback law $ u = \kappa_V(x) = K x $ equation $\dot{V} + l = 0 $ holds in this linear case.

The scalability of the ADMM-based DMPC scheme was examined by
varying the agent numbers $N$ between 10 and 140 and investigating the
iteration numbers $q_k$. 
To obtain representative results, 20 different initial value scenarios
were simulated for all cases of agent numbers $N$.  
For each initial value scenario, the maximum, the average, and the
minimum of all ADMM iteration numbers $q_k$ were determined. 
Figure~\ref{fig:scaleErr} shows the corresponding mean value and standard deviation
for the 20 scenarios plotted over the respective number $N$ of agents.
The maximum values of $q_k$ typically occurs in the first MPC step
$k=0$, since afterwards the warm-start
\eqref{eq:multipl_initialization} of the ADMM scheme leads to
significantly lower iteration numbers. The large standard deviation of
the maximum iteration number is caused by the random initial values
which directly effect the iterations required to satisfy the stopping
criterion \eqref{eq:ADMM_convcrit} in the first MPC step.
The minimum number of iterations shown in
Figure~\ref{fig:scaleErr} corresponds to the value that the ADMM
scheme tends to after the initial phase of the simulation time. 
The corresponding green plots indicate that this value and its
standard deviation is almost independent of the agent number $N$,
which demonstrates the good scaling of the ADMM-based DMPC scheme as
the number of agents increases.

\begin{figure}[t]
	\centering
%
%
\definecolor{mycolor1}{rgb}{0.00000,0.44700,0.74100}%
\definecolor{mycolor2}{rgb}{0.85000,0.32500,0.09800}%
\begin{tikzpicture}

\begin{axis}[%
width=0.951\fwidth,
height=\fheight,
at={(0\fwidth,0\fheight)},
scale only axis,
xmin=0,
xmax=160,
xlabel style={font=\color{white!15!black}},
xlabel={$N$},
ymin=0,
ymax=55,
ylabel style={font=\color{white!15!black}},
ylabel={$q_k$},
axis background/.style={fill=white},
xmajorgrids,
ymajorgrids,
legend style={at={(0.03,0.97)}, anchor=north west, legend cell align=left, align=left, draw=white!15!black}
]
\addplot [color=mycolor1, line width=1.2pt, draw=none, mark=x, mark options={solid, mycolor1}]
 plot [error bars/.cd, y dir = both, y explicit]
 table[row sep=crcr, y error plus index=2, y error minus index=3]{%
10	17	11.7652476655436	11.7652476655436\\
20	21.85	15.2876799382619	15.2876799382619\\
30	24.05	15.8296025152078	15.8296025152078\\
40	24.05	14.6196515186346	14.6196515186346\\
50	29.7	16.1899904099318	16.1899904099318\\
60	32.25	17.5285481429581	17.5285481429581\\
70	32.6	17.0645832061612	17.0645832061612\\
80	33.2	16.3758487124372	16.3758487124372\\
90	34.2	15.7600427464021	15.7600427464021\\
100	33.5	15.8529426124516	15.8529426124516\\
110	33.65	15.2876799382619	15.2876799382619\\
120	33.65	15.1493442623418	15.1493442623418\\
130	34.35	15.2152105335072	15.2152105335072\\
140	34.4	15.1601971935516	15.1601971935516\\
150	35.6	17.1537935040313	17.1537935040313\\
};
\addlegendentry{max}

\addplot [color=mycolor2, line width=1.2pt, draw=none, mark=x, mark options={solid, mycolor2}]
 plot [error bars/.cd, y dir = both, y explicit]
 table[row sep=crcr, y error plus index=2, y error minus index=3]{%
10	3.42333333333333	1.42685602940695	1.42685602940695\\
20	4.33083333333333	2.01552653721608	2.01552653721608\\
30	4.7725	2.36953295050075	2.36953295050075\\
40	4.89583333333333	2.54211241417759	2.54211241417759\\
50	5.88333333333333	2.46975269259355	2.46975269259355\\
60	6.25	2.41862896738449	2.41862896738449\\
70	6.37333333333333	2.39507535683374	2.39507535683374\\
80	6.48916666666667	2.2652361004194	2.2652361004194\\
90	6.72833333333333	2.12844217254577	2.12844217254577\\
100	6.81416666666667	2.00263185751279	2.00263185751279\\
110	6.86083333333333	1.97866748090729	1.97866748090729\\
120	7.03916666666667	2.3245591608728	2.3245591608728\\
130	7.1575	2.40248698507407	2.40248698507407\\
140	7.13833333333333	2.27014117086857	2.27014117086857\\
150	7.2925	2.35049127003831	2.35049127003831\\
};
\addlegendentry{avg}

\addplot [color=black!40!green, line width=1.2pt, draw=none, mark=x, mark options={solid, black!40!green}]
 plot [error bars/.cd, y dir = both, y explicit]
 table[row sep=crcr, y error plus index=2, y error minus index=3]{%
10	2	0	0\\
20	2	0	0\\
30	2	0	0\\
40	2.25	1.11803398874989	1.11803398874989\\
50	2.25	1.11803398874989	1.11803398874989\\
60	2.25	1.11803398874989	1.11803398874989\\
70	2.25	1.11803398874989	1.11803398874989\\
80	2.25	1.11803398874989	1.11803398874989\\
90	2.25	1.11803398874989	1.11803398874989\\
100	2.25	1.11803398874989	1.11803398874989\\
110	2.25	1.11803398874989	1.11803398874989\\
120	2.5	1.53896752812773	1.53896752812773\\
130	2.5	1.53896752812773	1.53896752812773\\
140	2.7	1.71985311490316	1.71985311490316\\
150	2.5	1.53896752812773	1.53896752812773\\
};
\addlegendentry{min}

\end{axis}
\end{tikzpicture}%
	\caption{Mean values and standard deviation of ADMM iteration numbers for 20 different $x_0$.}
	\label{fig:scaleErr}
\end{figure}
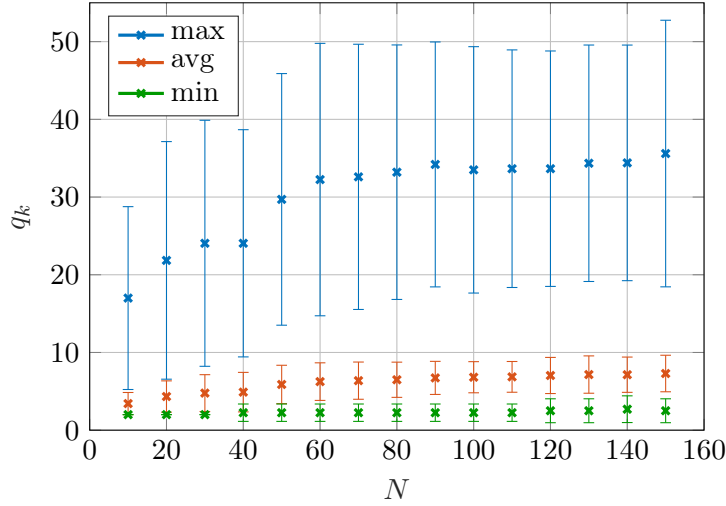

\section{Conclusions}
\label{sec:conclusions}

The distributed MPC scheme in this paper is presented for
continuous-time nonlinear, coupled systems.
The decoupling and distributed solution are based on dual
decomposition and using the alternating method of
multipliers in connection with a contracting stopping criterion that
limits the communication effort but renders the single DMPC solutions
suboptimal. 
Exponential stability is shown under the assumption of
linear convergence of the ADMM algorithm. Under the weaker assumption
of R-linear convergence, exponential stability can
still be ensured for a sufficiently large number of ADMM iterations. 
Two examples are used to illustrate the performance and scalability of
the DMPC scheme.

Current research concerns the modularization of the DMPC implementation for a flexible (re\hbox{-})configuration of coupled nonlinear systems. Further future work concerns the experimental evaluation of the ADMM-based DMPC scheme and its extension to further system classes.

\begin{appendix}

\section{Appendix}

\subsection{Proof of Lemma~\ref{lem:err_to_stop_crit}} \label{app:sec:proof_err_to_stop_crit}

To prove \eqref{eq:err_to_stop_crit}, the 
error norm $\|\Delta\xmpc^{q_k}_F(\cdot;\xsys_k)\|_{L^\infty}$
can be expanded using Minkowski's inequality
\begin{equation} \label{eq:bound_err_act_1}
\|\Delta\xmpc^{q_k}_F(\cdot;\xsys_k)\|_{L^\infty}
\le \|\Delta x^{q_k}(\cdot;\xsys_k)\|_{L^\infty} + 
\|\xi^{q_k}(\cdot;\xsys_k)\|_{L^\infty}
\end{equation}
and considering the terms $\Delta x^{q_k}:=\xmpc^{q_k}-\xmpc^*$ and 
$\xi^{q_k}:=\xmpc_F^{q_k}-\xmpc^{q_k}$ separately. 
Note that $\xi^{q_k}$ can be interpreted as the \emph{individual}
error of each subsystem. In the following lines, arguments will be
omitted where it is convenient in order to simplify notations.

A bound on $\xi^{q_k}$ can be
derived in view of the dynamics~\eqref{eq:traj_ind} and
\eqref{eq:traj_act} and the Lipschitz property in
Assumption~\ref{ass:lipschitz_DMPC} 
\begin{align} 
\nonumber
\|\xi^{q_k}(\tau;\xsys_k) \| &\le
\int_0^\tau \| f(\xmpc_F^{q_k},\umpc^{q_k},\hat\xmpc_F^{q_k}) - 
f(\xmpc^{q_k},\umpc^{q_k},v^{q_k}) \| \dd s
\\ \nonumber &\le
L_f \int_0^T \|\xi^{q_k}\| + \|\hat\xmpc_F^{q_k}-v^{q_k}\| \dd s
\\ \label{eq:bound_err_ind_1}&\le
L_f \int_0^T \|\xi^{q_k}\| + \|\hat\xmpc_F^{q_k}-\hat\xmpc^{q_k}\| 
+ \|\hat\xmpc^{q_k}-v^{q_k}\| \dd s
\end{align}
where the hat notation 
$\hat x=\left[[x_j]_{j\in\Ngets{i}}\right]_{i\in\Vcal}\in\Rbb^p$ is
introduced to achieve equivalence between $f$ and $F$.
Note that by assumption any solution of the system~\eqref{eq:dyn_sys} is
bounded for bounded controls, i.e.\ $x(t)=x(t;u,\xsys_k)\in X$ with
$X$ being compact for all $u(t)\in U$ and $\xsys_k\in\Gamma_\alpha$,
which implies the existence of a finite Lipschitz constant
$L_f<\infty$ in the second line of \eqref{eq:bound_err_ind_1}. The
third line in \eqref{eq:bound_err_ind_1} eventually follows from the triangle inequality.
The second term under the integral can be bounded
by\,\footnote{\label{foo:norms}%
  The Euclidean and supremum norms for a vector $a\in\Rbb^n$ are related
  by $\|a\|\le\sqrt{n}\|a\|_\infty$ and $\|a\|_\infty\le\|a\|$.
}
\begin{equation}
\label{eq:bound_err_ind_2}
\|\hat\xmpc_F^{q_k}-\hat\xmpc^{q_k}\| \le \sqrt{p} \|\hat\xmpc_F^{q_k}-\hat\xmpc^{q_k}\|_\infty
\le \sqrt{p} \|\xmpc_F^{q_k}-\xmpc^{q_k}\|_\infty \le \sqrt{p} \|\xmpc_F^{q_k}-\xmpc^{q_k}\|
= \sqrt{p} \| \xi^{q_k} \|
\end{equation}
by taking advantage of the fact that the stacked vectors
$\hat\xmpc_F^{q_k}$ and $\hat\xmpc^{q_k}$ exclusively consist of elements of
$\xmpc_F^{q_k}$ and $\xmpc^{q_k}$. The third term in \eqref{eq:bound_err_ind_1}
can be bounded in a similar manner using the stacked vector 
$\hat z=\left[[z_j]_{j\in\Ngets{i}}\right]_{i\in\Vcal}$ and the 
stopping criterion~\eqref{eq:ADMM_convcrit} in connection with the multiplier
update~\eqref{eq:ADMM:update_mu}, i.e.\ 
\begin{multline}
\label{eq:bound_err_ind_3}
\|\hat\xmpc^{q_k}-v^{q_k}\| \le \sqrt{p} \|\hat\xmpc^{q_k}-v^{q_k}\|_\infty \le
\sqrt{p} \left(\|\hat z^{q_k}-\hat\xmpc^{q_k}\|_\infty + \|\hat z^{q_k}-v^{q_k}\|_\infty\right)
\\ \le \sqrt{p} \left( \|z^{q_k}-\xmpc^{q_k}\| + \|\hat z^{q_k}-v^{q_k}\| \right)
\le \frac{2\sqrt{p}}{\rho} N  d\|\xsys_k\| \,.
\end{multline}
Inserting both bounds in \eqref{eq:bound_err_ind_1}, applying
Gronwall's inequality, and taking the $L^\infty$-norm eventually leads to
\begin{align} 
\label{eq:bound_err_ind_4}
\|\xi^{q_k}(\cdot;\xsys_k) \|_{L^\infty} &\le D'd\|\xsys_k\|
\end{align}
with $D'=\frac{2\sqrt{p}}{\rho} N L_f T e^{L_f(1+\sqrt{p})T}$.
The other term in \eqref{eq:bound_err_act_1} is treated in a similar
manner
\begin{align} \nonumber
\|\Delta x^{q_k}(\tau;\xsys_k)\| &\le
\int_0^T \|f(\xmpc^{q_k},\umpc^{q_k},v^{q_k}) - f(\xmpc^*,\umpc^*,\hat x^*) \| \dd s
\\ \label{eq:bound_err_act_2}
&\le L_f \int_0^T 
\|\Delta x^{q_k}\| + \|v^{q_k}-\hat x^*\| + \|\Delta u^{q_k}\| \dd s
\end{align}
with $\Delta u^{q_k}:=\umpc^{q_k}-\umpc^*$. 
The $\|v^{q_k}-\hat x^*\|$-term can be further bounded using the
triangle inequality and \eqref{eq:bound_err_ind_3}
\begin{align} \label{eq:bound_err_act_7}
\|v^{q_k}-\hat x^*\| &\le \sqrt{p} \left( \|v^{q_k}-\hat x^{q_k}\|_\infty +
  \|\hat x^{q_k}-\hat x^*\|_\infty\right)
\le 
\sqrt{p} \left(\frac{2}{\rho} N d\|\xsys_k\| + \|\Delta\xmpc^{q_k}\|\right) \,.
\end{align}
The last integral term in \eqref{eq:bound_err_act_2} 
can be expressed in terms of the optimal
feedback laws~\eqref{eq:optcontrlaw_centralized} and
\eqref{eq:optcontrlaw_distributed} and using the relation 
$\kappa^*(\cdot;\xsys_k)=\kappa(\cdot;z^*,\mu^*,\xsys_k)$
\begin{align} \nonumber
\int_0^\tau \|\Delta u^{q_k}(s;\xsys_k) \| \dd s
&=
\int_0^\tau\|\kappa(\xmpc^{q_k}(s;\xsys_k);z^{q_k-1},\mu^{q_k-1},\xsys_k) - 
\kappa(\xmpc^*(s;\xsys_k);z^*,\mu^*,\xsys_k) \|\,\dd s
\\ \nonumber &\le 
L_\kappa\int_0^\tau \|\Delta \xmpc^{q_k}\| + \|\Delta z^{q_k-1}\| 
+ \|\Delta \mu^{q_k-1}\| \,\dd s
\\ &\le \nonumber
\frac{2L_\kappa\tau}{1-C} \left\|\begin{matrix}
z^{q_k}-z^{q_k-1} \\ \mu^{q_k}-\mu^{q_k-1}
\end{matrix} \right\|_{L^\infty} +
L_\kappa\int_0^\tau \|\Delta \xmpc^{q_k}\|\dd s
\\ &\le \label{eq:bound_err_act_3}
\frac{2L_\kappa\tau}{1-C} N d \|\xsys_k\| +
L_\kappa \int_0^\tau \|\Delta \xmpc^{q_k}\| \,\dd s
\end{align}
Note that the states $(\xmpc^{q_k},\xmpc^*)$ are defined on a compact set 
as mentioned above. The same holds for 
$(z^{q_k-1},z^*)$ and the multipliers $(\mu^{q_k-1},\mu^*)$, 
cf.\ in particular Assumption~\ref{as:linear_convergence}.
Hence, the local Lipschitz property of $\kappa$ implies that there
exists a Lipschitz constant $L_\kappa<\infty$. 
The last two lines in \eqref{eq:bound_err_act_3} follows from
the linear convergence
property~\eqref{eq:linear_convergence} that implies
\begin{equation}
\label{eq:delta_linconv}
\begin{aligned}
\left\|\begin{matrix} z^{q_k-1}-z^* \\ \mu^{q_k-1}-\mu^* \end{matrix} \right\|_{L^\infty}
&\le \frac{1}{1-C}
\left\|\begin{matrix} z^{q_k}-z^{q_k-1} \\ \mu^{q_k}-\mu^{q_k-1} \end{matrix} \right\|_{L^\infty}
\end{aligned}
\end{equation}
together with the stopping criterion
\eqref{eq:ADMM_convcrit}.
Coming back to \eqref{eq:bound_err_act_2} eventually gives
\begin{align}
\label{eq:bound_err_act_10}
\|\Delta\xmpc^{q_k}(\tau;\xsys_k)\| \le
K_1\tau d \|\xsys_k\|
+ 
K_2\int_0^\tau \|\Delta\xmpc^{q_k}(s;\xsys_k)\|\dd s
\le
K_1 \tau e^{K_2\tau} d \|\xsys_k\|
\end{align}
with $K_1 = (L_f\frac{2\sqrt{p}}{\rho} +\frac{2L_\kappa}{1-C}) N$,  
$K_2=L_f(1+\sqrt{p}+L_\kappa)$, and using Gronwall's inequality.
Finally, the bound on the error \eqref{eq:bound_err_act_1} 
follows from the $L^\infty$-norm of \eqref{eq:bound_err_act_10}
together with the bound 
\eqref{eq:bound_err_ind_4} on the individual error 
\begin{equation}
\label{eq:bound_err_act_11}
\|\Delta\xmpc^{q_k}_F(\cdot;\xsys_k) \|_{L^\infty}
\le D d \|\xsys_k\|
\end{equation}
with $D= D'+ K_1 T e^{K_2 T}$, which
completes the proof of the lemma.

\subsection{Helpful Bounds}
\label{sec:helpfulbounds}

\ifthenelse{\equal{0}{1}}{
The (assumed) continuous differentiability of the system function
\eqref{eq:dyn_subsys} allows one to derive a useful bound on the difference between
individual state trajectories $\Delta\xmpc(\tau;\xsys_k):=\xmpc^q(\tau;\xsys_k)-\xmpc^{q-1}(\tau;\xsys_k)$
that are obtained for two subsequent solutions
$w^{q-1}(\tau;\xsys_k)$, $w^{q}(\tau;\xsys_k)$ of the
subproblems~\eqref{eq:cost:MPClocal}. 
With $\Gamma_\alpha$ being compact and the assumed boundedness of the
optimal trajectories $w^{q}(\tau;\xsys_k)\in W$ and
$\xmpc^q(\tau;\xsys_k)\in X$ with $W$ and $X$ being compact, there
exists a finite Lipschitz constant $L_f>0$ such that the following estimate
holds (omitting arguments):
\begin{align} \nonumber
\|\Delta x(\tau;\xsys_k)\| &\le
\int_0^\tau \|f(\xmpc^q,\umpc^q,v^q)-f(\xmpc^{q-1},\umpc^{q-1},v^{q-1})\| \dd s
\\ \nonumber &\le
L_f \int_0^\tau  \|\Delta x(s;\xsys_k)\| + \|\Delta w(s;\xsys_k)\| \dd
s
\\ &\le \label{eq:bound_state_diff_1}
L_f\tau e^{L_f\tau} \|\Delta w(\cdot;\xsys_k)\|_{L^\infty}
\end{align}
with $\Delta
w(\tau;\xsys_k):=w^q(\tau;\xsys_k)-w^{q-1}(\tau;\xsys_k)$. The last
line is follows from Gronwalls inequality after using the supremum
norm for the control part over $\tau\in[0,T]$ under the integral.  
As \eqref{eq:bound_state_diff} holds for all $\tau\in[0,T]$, we have
\begin{equation} \label{eq:bound_state_diff_2}
\|\xmpc^q(\cdot;\xsys_k)-\xmpc^{q-1}(\cdot;\xsys_k)\|_{L^\infty} \le
D_x \|w^q(\cdot;\xsys_k)-w^{q-1}(\cdot;\xsys_k)\|_{L^\infty} \,, \quad
D_x = L_f T e^{L_f T}\,.
\end{equation}
}{}

Some bounds on the optimal state trajectory $\xmpc^*(\tau;\xsys_k)$
can be derived using Assumption~\ref{ass:lipschitz_DMPC}.
Note that for $\xsys_k\in\Gamma_\alpha$, the optimal trajectory
satisfies $\xmpc^*(\tau;\xsys)\in \Gamma_\alpha$
for all $\tau\in[0,T]$, which follows from~\eqref{eq:cost_decrease_centr_MPC_terminal_cost}.
Considering the equilibrium $F(0,0)=0$ as well as
$\kappa^*(0;\xsys_k)=0$ for the optimal feedback~\eqref{eq:optcontrlaw_centralized},  
we have the Lipschitz estimates 
$\|F(x,u)\|=\|[f_i(x_i,u_i, [x_j]_{j\in\Ngets{i}})]_{i\in\Vcal}\|
\le L_f(\|x\|+\|u\|)$ and $\|\kappa^*(x;\xsys_k)\|\le
L_{\kappa^*}\|x\|$ for all $x\in X_\alpha$, $u\in U$ with some
finite Lipschitz constants $L_f,L_{\kappa^*}>0$.
Using Gronwall's inequality, the optimal state trajectory can be
bounded by
\begin{align} 
\|\xmpc^*(\tau; \xsys_k) \| &\le
\|\xsys_k\| + \int_{0}^\tau \| F( \xmpc^*(s;\xsys_k),\umpc^*(s;\xsys_k))\| 
\dd s \nonumber \\
&\le 
\|\xsys_k\| + L_f \int_{0}^\tau \|\xmpc^*(s;\xsys_k)\| + 
\|\kappa^*(\xmpc^*(s;\xsys_k);\xsys_k) \| \dd s \nonumber \\
&\le
\|\xsys_k\| + L_f(1+L_{\kappa^*}) \int_{0}^\tau
\|\xmpc^*(s;\xsys_k)\| \dd s
\nonumber \\
&\le \label{eq:app:upper_bound_xopt}
\|\xsys_k\| e^{ \hat{L} \tau}
\end{align}
with $\hat{L} := L_f(1+L_{\kappa^*})$.
A lower bound is derived in a similar manner using an inverse
formulation of Gronwall's inequality \cite{Gollwitzer1969}
\begin{align} \label{eq:app:lower_bound_xopt}
\|\xmpc^*(\tau; \xsys_k) \| &\ge
\|\xsys_k\| - \int_{0}^\tau \| F( \xmpc^*(s;\xsys_k),\umpc^*(s;\xsys_k))\| 
\dd s \nonumber \\ \nonumber &\ge
\norm{x_k} - \hat L \int_{0}^\tau \| \xmpc^*(s;\xsys_k)\| \dd s
\\ &\ge \|\xsys_k\| e^{ -\hat{L} \tau } \,.
\end{align}
The estimates \eqref{eq:app:upper_bound_xopt} and
\eqref{eq:app:lower_bound_xopt} can be used for an 
upper bound on the optimal cost~\eqref{eq:Jopt}
\begin{align}
\label{eq:app_J_upper_bound}
J^*(\xsys_k) &\le \nonumber
M_V \|\xmpc^*(T;\xsys_k)\|^2 + M_l \int_0^T 
\|\xmpc^*(\tau;\xsys_k)\|^2 + \|\umpc^*(\tau;\xsys_k)\|^2
\dd\tau
\\ &\le \nonumber
M_V e^{2\hat L T} \|\xsys_k\|^2 + M_l \|\xsys_k\|^2 \int_0^T 
e^{2\hat L\tau} + L_{\kappa^*}^2 e^{2\hat L T} \dd\tau
\\ &=
M_J \|\xsys_k\|^2
\end{align}
with $M_J:= M_V e^{2\hat L T}+\frac{M_L}{2\hat L}\left(e^{2\hat L T}(1+L_{\kappa^*}^2)-1-L_{\kappa^*}^2 \right)$
as well as for a lower bound
\begin{align} \nonumber
\label{eq:app_J_lower_bound}
J^*(\xsys_k) &\ge
m_l \int_0^T \|\xmpc^*(\tau;\xsys_k)\|^2 \dd\tau
\\ &\ge \nonumber
m_l \|\xsys_k\|^2 \int_0^T e^{-2\hat L \tau}
\\ &= 
m_J \|\xsys_k\|^2
\end{align}
with $m_J:=\frac{m_l}{2\hat L} \left(1-e^{-2\hat L T} \right)$.
In addition, the integral in 
\eqref{eq:cost_decrease_centr_MPC_terminal_cost} can be lower bounded by
\begin{align} \nonumber
\int_0^{\Delta t} l(\xmpc^*(\tau;\xsys_k),\umpc^*(\tau;\xsys_k)) \dd\tau
&\ge 
m_l \int_0^{\Delta t} \|\xmpc^*(\tau;\xsys_k)\|^2 \dd\tau
\\ \label{eq:app_intdt_lower_bound}
&\ge a J^*(\xsys_k) \,, \quad
a:= \frac{m_l (1-e^{-2\hat L\Delta t})}{2\hat L M_J}
\end{align}
with $0<a\le 1$.

\end{appendix}
\printbibliography

\end{document}